\documentclass[11pt]{article}
\title{\vspace{-1cm} \linespread{1.15} \bfseries \large ADELIC DESCENT FOR K-THEORY}
\author{\MakeUppercase Hyungseop Kim}
\date{}

\let\originallhook=\lhook

\usepackage{amsthm,amssymb,amsmath}  
\usepackage{mathabx}
\usepackage[T1]{fontenc}
\usepackage{lmodern}
\usepackage{xcolor}
\usepackage[all,matrix,color,arrow,pdf]{xy}  
\usepackage{graphicx}
\usepackage{mathtools}
\usepackage{mathrsfs}
\usepackage{geometry}
\usepackage{tikz}  
\usetikzlibrary{calc, intersections}
\usepackage{enumerate}
\usepackage{pdfpages}
\usepackage[utf8]{inputenc}
\usepackage{longtable}
\usepackage{setspace}
\usepackage{hyperref}
\usepackage{bbm}
\usepackage{titlesec}
\usepackage{MnSymbol}
\usepackage{eucal} 
\usepackage{tikz-cd}

\let\lhook=\originallhook

\newcommand\sbullet[1][.5]{\mathbin{\vcenter{\hbox{\scalebox{#1}{$\bullet$}}}}} 

\DeclareMathAlphabet{\mathpzc}{OT1}{pzc}{m}{it}

\def\Map{\operatorname{Map}}
\def\map{\operatorname{map}}

\def\GL{\operatorname{GL}}

\def\Sp{\operatorname{Sp}}
\def\Spec{\operatorname{Spec}}

\def\op{\operatorname{op}}

\def\N{\operatorname{N}}

\def\colim{\operatorname{colim}}

\def\Ring1{\mathsf{Ring}1}
\def\CRing1{\mathsf{CRing}1}
\def\Mod{\text{Mod}}

\def\QCoh{\text{QCoh}}

\def\CAlg{\text{CAlg}}

\def\Sch{\text{Sch}}

\def\Frac{\text{Frac}}

\def\bbN{\mathbb{N}}
\def\bbZ{\mathbb{Z}}

\def\Frac{\operatorname{Frac}}

\def\+1{\xrightarrow{+1}}

\def\psh{_{\ast}}
\def\inv{^{-1}}
\def\pb{^{\ast}}
\def\ush{^{!}}
\def\lsh{_{!}}

\def\Supp{\text{Supp}}

\def\Gr{\text{Gr}}

\def\isomto{\underset{\sim}{\to}}

\def\loc{\text{loc}}

\def\fg{\text{fg}}

\def\Cat{\text{Cat}}

\def\h{\text{h}}
\def\perf{\text{perf}}
\def\ex{\text{ex}}
\def\st{\text{st}}
\def\Perf{\text{Perf}}

\def\fgProj{\text{Proj}^{\fg}}

\def\Prl{\text{Pr}^{\text{L}}}

\def\st{\text{st}}

\def\Sp{\text{Sp}}

\def\N{\text{N}}
\def\Fun{\text{Fun}}
\def\fib{\text{fib}}
\def\cof{\text{cof}}

\def\Ind{\text{Ind}}

\def\co{\text{co}}

\def\CFib{\text{CFib}}
\def\loc{\operatorname{loc}}

\def\Cpl{\operatorname{Cpl}}

\def\Shv{\operatorname{Shv}}

\newgeometry{vmargin={23mm}, hmargin={19mm} }

\theoremstyle{definition}
\newtheorem{example}{Example}[subsection]  
\newtheorem{definition}[example]{Definition}

\newtheorem{proposition}[example]{Proposition}
\newtheorem{lemma}[example]{Lemma}

\newtheorem{theorem}[example]{Theorem}
\newtheorem{corollary}[example]{Corollary}

\newtheorem{remarkn}[example]{Remark} 

\newtheorem*{ack}{Acknowledgements}

\theoremstyle{remark}

\renewenvironment{abstract}{\noindent\begin{center}\begin{minipage}{0.85\linewidth}\small{\scshape Abstract.}}{\end{minipage}\end{center}}

\begin{document}
\maketitle
\begin{abstract}
We prove an adelic descent result for localizing invariants: for each Noetherian scheme $X$ of finite Krull dimension and any localizing invariant $E$, e.g., algebraic K-theory of Bass-Thomason, there is an equivalence $E(X)\simeq \lim E(A^{\sbullet}_{\text{red}}(X))$, where $A^{\sbullet}_{\text{red}}(X)$ denotes Beilinson's semi-cosimplicial ring of reduced adeles on $X$. We deduce the equivalence from a closely related cubical descent result, which we prove by establishing certain exact sequences of perfect module categories over adele rings. 
\end{abstract}

{
  \let\clearpage\relax
  \small
  \tableofcontents
}

\section{Introduction}
Let $X$ be a Noetherian scheme of finite Krull dimension $n$. Beilinson's construction \cite{beilinson} of higher adeles on Noetherian schemes produces the semi-cosimplicial ring $A^{\sbullet}_{\text{red}}(X)$ of reduced adeles on $X$, whose associated complex of abelian groups computes the cohomology of $\mathscr{O}_{X}$. Each ring $A^{r}_{\text{red}}(X)$ of reduced adeles decomposes into a product of adele rings $A(i_{0},...,i_{r})$ indexed by $0\leq i_{0}<\cdots<i_{r}\leq n$, i.e., subsets of $[n]$ of cardinality $r+1$. Moreover, the association $\{i_{0},...,i_{r}\}\mapsto A(i_{0},...,i_{r})$ defines a functor $A$ on $\mathcal{P}([n])\backslash\emptyset$, i.e, an $n$-cubical diagram $A$ of rings without the initial vertex (see Remark \ref{rmk:cubicaladeles} (1)). Our goal in this paper is to prove the following adelic descent result for nonconnective algebraic K-theory spectra (or more generally for any localizing invariant of stable $\infty$-categories):

\begin{theorem}[Theorem \ref{thm:adelicdescent}] \label{thm:main}
Let $X$ be a Noetherian scheme of finite Krull dimension $n$. Also, let $A^{\sbullet}_{\text{red}}(X)$ and $A(-)$ be the semi-cosimplicial and cubical (without the initial vertext) diagram of adele rings on $X$ respectively (see Remarks \ref{rmk:functoriality} and \ref{rmk:cubicaladeles}). Then, we have equivalences of (nonconnective) algebraic K-theory spectra 
\begin{align}
K(X)\simeq\lim_{[r]\in(\Delta_{s})_{\leq n}}K(A^{r}_{\text{red}}(X))~~\text{and}   \label{formula:descent1} \\
K(X)\simeq\lim_{0\leq i_{0}<\cdots<i_{r}\leq n}K(A(i_{0},...,i_{r})).   \label{formula:descent2}
\end{align}
\end{theorem}

\begin{remarkn}
Theorem \ref{thm:main} remains valid if we replace the algebraic K-theory functor $K$ by any localizing invariant $E:\Cat^{\ex}\to\mathcal{T}$ valued in a stable $\infty$-category $\mathcal{T}$. In fact, the proof of Theorem \ref{thm:adelicdescent} does not use any properties specific to $K$ except that $K$ is a localizing invariant of small stable $\infty$-categories. 
\end{remarkn}

For the case of curves, Theorem \ref{thm:main} is closely related to Weil's description of vector bundles on $X$. Suppose $X = \Spec R$ for a Dedekind ring $R$ which is not a field. Then, $A(0) = F$ is the field of fractions of $R$, $A(1) = O$ is the ring of integral adeles $\prod_{\mathfrak{p}\in (\Spec R)_{0}}R_{\mathfrak{p}}^{\wedge}$ (where the product is taken over maximal ideals of $R$), and $A(01) = A = F\otimes_{R}O$ is the ring of finite adeles. Moreover, $A^{\sbullet}_{\text{red}}(X)$ takes the form of $F\times O\rightrightarrows A$. Weil's adelic uniformization theorem implies we have an equivalence between the (1-)groupoid $\text{BGL}_{r}(R)$ of rank $r$ vector bundles on $\Spec R$ and the double quotient groupoid $[\GL_{r}(F)\backslash \GL_{r}(A)/\GL_{r}(O)]$ (see \cite[Corollary 3.38 and 3.39]{adelic} for details and generalizations to Noetherian schemes). On objects, the equivalence sends each isomorphism class of finite projective $R$-module $M$ of rank $r$ to the double coset represented by $(\phi_{\eta}|_{F_{p}^{\wedge}}\circ\phi_{p}|_{F^{\wedge}_{p}}^{-1})_{p\in X_{0}}\in\GL_{r}(A)$, where $\phi_{\eta}$ is a trivialization of $M$ on a nonempty open subset of $X$ (hence gives a trivialization of $F\otimes_{R}M$) and each $\phi_{p}$ is a trivialization of $M^{\wedge}_{p}$ at the closed point $p$ of $X$. In particular, each finite projective $R$-module is obtained by gluing finite projective (in fact, finite free) modules over $F$ and $O$ which are isomorphic to each other over $A$ after base change. From this, we know there is an equalizer diagram $\pi_{0}\fgProj(R)\to\pi_{0}\fgProj(F)\times\pi_{0}\fgProj(O)\rightrightarrows\pi_{0}\fgProj(A)$ of commutative monoids\footnote{Here, for each ring $R$, $\fgProj(R)$ stands for the (nerve of the) category of finite projective $R$-modules equipped with the monoidal structure given by direct sums. Thus, the set $\pi_{0}\fgProj(R)$ of isomorphism classes of finite projective $R$-modules admits a commutative monoid structure.}. After group-completion, we obtain a sequence $K_{0}(R)\to K_{0}(F)\oplus K_{0}(O)\to K_{0}(A)$ of abelian groups. Note that in general, this sequence does not realize $K_{0}(R)$ as a kernel of the second map, as the group-completion functor (which is a left adjoint functor) does not preserve limits in general. Nonetheless, we can realize this sequence as a part of a long exact sequence through our descent result. By Theorem \ref{thm:main}, we have a pullback square
\begin{equation}\label{diag:curve}
\begin{tikzcd}
K(R) \ar[r] \ar[d] & K(F) \ar[d] \\
K(O) \ar[r] & K(A)
\end{tikzcd}
\end{equation} 
of spectra, and hence we have a long exact Mayer-Vietoris sequence 
\begin{align*}
\cdots\to K_{i}(R)\to K_{i}(F)\oplus K_{i}(O)\to K_{i}(A)\to K_{i-1}(R)\to\cdots~~~(i\in\bbZ)
\end{align*}
of abelian groups. Around degree $i=0$, we have an exact sequence $\cdots \to K_{1}(A)\to K_{0}(R)\to K_{0}(F)\oplus K_{0}(O)\to K_{0}(A)\to 0$ which extends the previous sequence of abelian groups obtained from the Weil uniformization theorem. \\
\indent Another motivation for our result is the following adelic descent theorem of \cite{adelic} for perfect modules. Recall that for a Noetherian scheme $X$, Beilinson's construction indeed provides us the cosimplicial ring $A^{\sbullet}(X)$ of adeles on $X$ whose dual normalization is $A^{\sbullet}_{\text{red}}(X)$ \cite[Proposition 5.1.3]{huber}. 
\begin{theorem}[\cite{adelic}, Theorem 3.1] \label{thm:perfectdescent}
Let $X$ be a Noetherian scheme. Then, there is an equivalence of symmetric monoidal stable $\infty$-categories $\Perf(X)\simeq \lim_{[r]\in\Delta}\Perf(A^{r}(X))$ in $\CAlg(\Cat^{\perf})$. 
\end{theorem}
As localizing invariants (and in particular the algebraic K-theory functor $K$) do not preserve limits, in fact even pullbacks in general, we cannot deduce our descent result for K-theory spectra directly from Theorem \ref{thm:perfectdescent}. Instead, we follow a strategy which is more suited to investigate descent results for localizing invariants, and independent of the proof of Theorem \ref{thm:perfectdescent} given in \cite{adelic}.  \\
\indent We approach Theorem \ref{thm:main} as follows. Through a comparison between cubical and semi-cosimplicial limits (Corollary \ref{cor:fiberlim}), we deduce the semi-cosimplicial descent (\ref{formula:descent1}) from the cubical descent (\ref{formula:descent2}). In order to prove (\ref{formula:descent2}), we introduce auxiliary stable subcategories $\Perf_{\leq i}(A(T))$ of the $\infty$-category $\Perf(A(T))$ of perfect modules\footnote{Here, we set $\Perf(A(\emptyset)) = \Perf(X)$.} over the adele ring $A(T)$ for each $0\leq i\leq n$ and $T\subseteq[n]$ (Definition \ref{def:perf}), and prove that we have exact sequences $\Perf_{\leq i-1}(A(T))\to\Perf_{\leq i}(A(T))\to\Perf_{\leq i}(A(T\sqcup\{i\}))$ of small stable $\infty$-categories for each $T\subseteq[i-1]$ (Proposition \ref{prop:exactseqadele}). When $i=n$, the image of the second map by $K$ in this exact sequence recovers the $n$-cubical diagram $T\mapsto K(A(T))$ of (\ref{formula:descent2}). For $n=2$ (i.e., the case of surfaces), this $2$-cube is obtained as an image of the right side $2$-cube by $K$ in the following diagram of small stable $\infty$-categories:
\begin{equation}\label{diag:2cubes}
\begin{tikzcd}[row sep={35,between origins}, column sep={60,between origins}]
\Perf_{\leq 1}(X)\ar[rr] \ar[dr,swap] \ar[dd,swap] &&
\Perf(X) = \Perf_{\leq 2}(X)\arrow[rr] \arrow[dr,swap] \arrow[dd,swap] &&
  \Perf_{\leq 2}(A(2)) \arrow[dd,swap] \arrow[dr] \\
& \Perf_{\leq 1}(A(1)) \ar[rr,crossing over] &&
\Perf_{\leq 2}(A(1)) \arrow[rr,crossing over] &&
  \Perf_{\leq 2}(A(12)) \arrow[dd] \\
\Perf_{\leq 1}(A(0)) \ar[rr] \ar[dr,swap] &&
\Perf_{\leq 2}(A(0)) \arrow[rr] \arrow[dr,swap] && \Perf_{\leq 2}(A(02)) \arrow[dr,swap] \\
& \Perf_{\leq 1}(A(01)) \ar[rr] \ar[uu,<-,crossing over] &&
\Perf_{\leq 2}(A(01)) \arrow[rr] \arrow[uu,<-,crossing over]&& \Perf_{\leq 2}(A(012)).
\end{tikzcd}
\end{equation}
After applying $K$, the four horizontal sequences in the diagram (\ref{diag:2cubes}) become fiber sequences of spectra. Thus, by \cite[1.2.4.15]{ha} (see Proposition \ref{prop:fiberlimcrit}) the $n$-cube $T\mapsto K(A(T))$ is a limit diagram precisely when the $(n-1)$-cube $T\mapsto K(\Perf_{\leq n-1}(A(T)))$ (where $T\subseteq[n-1]$) is a limit diagram. For $n=2$, this $1$-cube is an image of the leftmost square of the diagram (\ref{diag:2cubes}) by $K$. Using the exact sequences of Proposition \ref{prop:exactseqadele}, we can repeat this procedure on each $i$-cube $T\mapsto K(\Perf_{\leq i}(A(T)))$ for all $0\leq i\leq n$ until we reach $i=0$. For our $n=2$ case, the leftmost square of the diagram (\ref{diag:2cubes}) fits into the following new diagram (as the right side $1$-cube) whose rows are exact sequences of small stable $\infty$-categories:
\begin{equation}\label{diag:1cubes}
\begin{tikzcd}
\Perf_{\leq0}(X) \ar[r] \ar[d]_{\sim}^{i=0} & \Perf_{\leq 1}(X) \ar[r] \ar[d] & \Perf_{\leq1}(A(1)) \ar[d]\\
\Perf_{\leq 0}(A(0)) \ar[r] & \Perf_{\leq 1}(A(0)) \ar[r] & \Perf_{\leq 1}(A(01)).
\end{tikzcd}
\end{equation}  
Note that after applying the functor $K$, the right side $1$-cube of the diagram (\ref{diag:1cubes}) takes the form of the square (\ref{diag:curve}) for curves. By Proposition \ref{prop:exactseqadele} and $\Perf_{\leq -1}(X)=0$ (or by \cite[Theorem 2.6.3]{tt}), the left vertical arrow $\Perf_{\leq 0}(X)\to\Perf_{\leq 0}(A(0))$ of the diagram (\ref{diag:1cubes}) is an equivalence. Thus, the $0$-cube $K(\Perf_{\leq 0}(X))\to K(\Perf_{\leq 0}(A(0)))$ is a limit diagram, and we know the $1$-cube $T\mapsto K(\Perf_{\leq 1}(A(T)))$ and the original $2$-cube $T\mapsto K(A(T))$ are limit diagrams. \\
\indent In \cite{tt}, Thomason showed that algebraic K-theory satisfies Zariski descent for qcqs (i.e., quasicompact quasiseparated) schemes through Zariski excision \cite[Theorem 8.1]{tt}. There, he first established exact sequences $\Perf_{Z}(X)\to\Perf(X)\to\Perf(U)$ of perfect modules for each quasicompact open embedding $U\hookrightarrow X$ with complement $Z = X\backslash U$. Then, he used the equivalence $\Perf_{Z}(X)\simeq \Perf_{\pi\inv Z}(Y)$ for each flat morphism $Y\xrightarrow{\pi}X$ of qcqs schemes which is an isomorphism over $Z$ in order to prove the equivalence $K(X)\simeq K(Y)\times_{K(\pi\inv U)}K(U)$ as in the situation of diagram (\ref{diag:1cubes}). Our proof of Theorem \ref{thm:main} applies this Thomason-Trobaugh argument to cubical diagrams of perfect modules at each induction step to deduce the descent result. Note that Thomason's approach more generally proves that localizing invariants satisfy Nisnevich excision, and hence satisfy Nisnevich descent for qcqs schemes (\cite[3.7.5.1]{sag}, see also \cite[Proposition 5.15]{etalek} for qcqs spectral algebraic spaces). Although localizing invariants do not satisfy the more useful \'etale descent property in general, Clausen and Mathew proved that localizing invariants valued in $L_{n}^{f}$-local spectra satisfy \'etale descent on ($\mathbb{E}_{2}$-)spectral algebraic spaces \cite[Theorem 5.39]{etalek} (see \cite[Theorem 7.14]{etalek} for the \'etale hyperdescent result under finiteness conditions), e.g., $T(n)$-localized algebraic $K$-theory $L_{T(n)}K$ satisfies \'etale descent for all $n$, which generalizes Thomason's result for $L_{T(1)}K$. Theorem \ref{thm:main}, although not a descent result for a particular Grothendieck topology, provides a descent result for localizing invariants for an adelic resolution $A^{\sbullet}_{\text{red}}(X)$ of each Noetherian scheme $X$ of finite Krull dimension, allowing one to understand $K(X)$ via K-theory of adele rings and maps between them. 

\begin{ack}
The author would like to thank his advisor Michael Groechenig for his suggestion of the problem and support through numerous discussions, without which this project would not have led anywhere. I also would like to thank Benjamin Antieau for helpful discussions through email, and Oliver Braunling for pointing out Balmer's paper \cite{balmer2}. The author was supported by the CMK Foundation. 
\end{ack}

\section{Categorical and algebraic backgrounds}
In this section we review and explain necessary backgrounds on $\infty$-categories and higher adeles. In \ref{subsec:stable} we review the notion of exact sequences of stable $\infty$-categories and localizing invariants following \cite{bgt}. In \ref{subsec:limit} we study cubical and semi-cosimplicial diagrams and their limits through Cartesian fibrations. As we use the language of sheaves of module spectra following \cite{sag}, we briefly recall some of their theory in \ref{subsec:modules}. Finally, we review semi-cosimplicial and cubical sheaves of adele rings, as well as modules over adele rings on Noetherian schemes in \ref{subsec:adeles}. 

\subsection{Stable $\infty$-categories and localizing invariants}
\label{subsec:stable}

Let $\Cat^{\ex}$ be the $\infty$-category of small stable $\infty$-categories and exact functors, and let $\Prl_{\st}$ be the $\infty$-category of presentable stable $\infty$-categories and left adjoint (i.e., colimit preserving) functors. The ind-completion construction $\Ind:\Cat^{\ex}\to\Prl_{\st}$ (as in \cite[5.3.5]{htt}) relates these two categories, and factors through the full subcategory of $\Prl_{\st}$ spanned by compactly generated stable $\infty$-categories. Let $\Cat^{\perf}$ be the full subcategory of $\Cat^{\ex}$ consisting of idempotent complete small stable $\infty$-categories. Then, the construction $(\Ind(-))^{\omega}:\Cat^{\ex}\to\Cat^{\perf}$ is well-defined and behaves as indempotent-completion (i.e., provides a left adjoint to the inclusion functor $\Cat^{\perf}\subseteq\Cat^{\ex}$). In fact, $\Ind$ induces an equivalence $\Ind:\Cat^{\perf}\to\Prl_{\st,\omega}$ from $\Cat^{\perf}$ onto the $\infty$-category $\Prl_{\st,\omega}$ of compactly generated (presentable) stable $\infty$-categories and compact left adjoint functors (i.e., those preserving compact objects, or equivalently those with filtered-colimit preserving right adjoints \cite[5.5.7.2]{htt}), whose inverse is given by the functor $(-)^{\omega}$ taking ($\omega$-)compact objects of each $\infty$-categories. Recall that both $\Cat^{\ex}$ and $\Cat^{\perf}$ admit all (small) limits and colimits, and the inclusion $\Cat^{\ex}\to\Cat_{\infty}$ preserves limits and filtered colimits (\cite[1.1.4]{ha} and \cite[4.25]{bgt}). Likewise, recall that $\Prl_{\st}$ has all (small) limits and colimits, and the inclusion $\Prl_{\st}\to\widehat{\Cat}_{\infty}$ preserves limits and (all) colimits (\cite[Proposition 2.4 and its proof]{galois}, see also \cite[5.5.3.13, 5.5.3.18]{htt}). Note that the inclusion $\Prl_{\st,\omega}\to\widehat{\Cat}_{\infty}$, although preserving colimits \cite[5.5.7.6, 5.5.7.7]{htt}, does not preserve limits in general; as noted in \cite[Section 2]{kel}, fiber products of compactly generated presentable stable $\infty$-categories may not be compactly generated.  \\
\indent Let $\mathcal{A}\to\mathcal{B}\to\mathcal{C}$ be a sequence in $\Prl_{\st}$. Recall that the sequence is called \emph{exact} if the composite functor is zero, $\mathcal{A}\to\mathcal{B}$ is fully faithful, and the induced functor $\mathcal{B}/\mathcal{A}\to\mathcal{C}$ is an equivalence. Here, $\mathcal{B}/\mathcal{A}$ is a cofiber of the functor $\mathcal{A}\to\mathcal{B}$ (i.e., a pushout of functors $\mathcal{A}\to\mathcal{B}$ and $\mathcal{A}\to 0$ in $\Prl_{\st}$), which can be described via Bousfield localization in $\Prl_{\st}$. In fact, the homotopy category of $\mathcal{B}/\mathcal{A}$ is equivalent to the Verdier quotient of the inclusion $\text{h}\mathcal{A}\to \text{h}\mathcal{B}$, and the sequence is exact precisely if the corresponding sequence of homotopy categories is an exact sequence of triangulated categories \cite[5.9-5.11]{bgt}. Now, we call a sequence $\mathcal{A}\to\mathcal{B}\to\mathcal{C}$ in $\Cat^{\ex}$ \emph{exact} if the resulting sequence $\Ind(\mathcal{A})\to\Ind(\mathcal{B})\to\Ind(\mathcal{C})$ in $\Prl_{\st}$ is exact in the previous sense. This is equivalent to the condition that the composite functor is zero, $\mathcal{A}\to\mathcal{B}$ is fully faithful, and the induced functor $\mathcal{B}/\mathcal{A}\to\mathcal{C}$ is an equivalence \emph{after} idempotent completion \cite[5.13]{bgt}. One can describe the cofiber $\mathcal{B}/\mathcal{A}$ in $\Cat^{\ex}$ intrinsically (i.e., without embedding into $\Ind(\mathcal{B})$) through Dwyer-Kan localization in a way compatible with Bousfield localization in $\Prl_{\st}$, and still $\h(\mathcal{B}/\mathcal{A})\simeq \h\mathcal{B}/\h\mathcal{A}$ holds. In fact, one has a description of the mapping space as a filtered colimit $\Map_{\mathcal{B}/\mathcal{A}}(\overline{b},\overline{c})\simeq\colim_{a\in\mathcal{A}_{/c}}\Map_{\mathcal{B}}(b,\cof(a\to c))$, where $\overline{b}$ and $\overline{c}$ denote images of $b,c\in\mathcal{B}$ respectively \cite[I.3.3]{ns}.   \\
\indent Let us briefly explain the notion of fiber (or kernel) categories of exact functors between stable $\infty$-categories, which will be useful in the description of split-exact sequences. 

\begin{proposition}\label{prop:infinityfiber}
Let $\mathcal{C}$ be a pointed $\infty$-category, and let $\mathcal{B}\xrightarrow{q}\mathcal{C}$ be a functor from an $\infty$-category $\mathcal{B}$. Fix any zero object $0:\Delta^{0}\to\mathcal{C}$ of $\mathcal{C}$. Then, the $\infty$-category $\fib(q) = \Delta^{0}\times_{\mathcal{C}}\mathcal{B}$ is equivalent to the full subcategory of $\mathcal{B}$ generated by the objects $\{b\in\mathcal{B}~|~q(b)\simeq 0\}$.
\begin{proof}
First, let us describe $\fib(q)$ explicitly in terms of quasicategories. Let $\mathcal{C}'$ be the full subcategory of $\mathcal{C}$ generated by zero objects. 
\begin{lemma}\label{lem:infinityfiber}
The inverse image simplicial set $q^{-1}(\mathcal{C}')\hookrightarrow \mathcal{B}$ is a quasicategory equivalent to a pullback $\fib(q) = \Delta^{0}\times_{\mathcal{C}}\mathcal{B}$ of $\infty$-categories. 
\begin{proof} 
As $\mathcal{C}'$ is a contractible Kan complex \cite[1.2.12.9]{htt}, the canonical map $\mathcal{C}'\to\Delta^{0}$ is a categorical equivalence, and hence its section $\Delta^{0}\to\mathcal{C}'$ given by $0$ is an equivalence. This induces an equivalence $\fib(q)\simeq \mathcal{C}'\times_{\mathcal{C}}\mathcal{B}$, and hence to prove the claim, it suffices to check the latter can be computed by a simplicial set $q^{-1}(\mathcal{C}')$. The inclusion $\mathcal{C}'\xhookrightarrow{\imath}\mathcal{C}$ of a subcategory is by definition an inner fibration, and it is also an isofibration, since any isomorphism $\imath(0')\to 0''$ in $\mathcal{C}$ from a zero object comes from an isomorphism $0'\to0''$ in $\mathcal{C}'$. Hence, $\mathcal{C}'\xhookrightarrow{\imath}\mathcal{C}$ is a categorical fibration, and the pullback $\infty$-category $\mathcal{C}'\times_{\mathcal{C}}\mathcal{B}$ is computed by the inverse image quasicategory $q^{-1}(\mathcal{C}')\hookrightarrow \mathcal{B}$.  
\end{proof}
\end{lemma}

\noindent By Lemma \ref{lem:infinityfiber}, it suffices to check that the quasicategory $q^{-1}(\mathcal{C}')$ is the full subcategory of $\mathcal{B}$ determined by the set of objects $\{b\in\mathcal{B}~|~q(b)\simeq 0\}$. The inverse image quasicategory $q^{-1}(\mathcal{C}')$ is the pullback of simplicial sets viewed as set-valued presheaves on $\Delta$. As $\mathcal{C}'\hookrightarrow\mathcal{C}$ is a full subcategory, it is a full simplicial subset \cite[tag 01CU]{kerodon}, and hence its inverse image $q^{-1}(\mathcal{C}')\hookrightarrow\mathcal{B}$ is a full simplicial subset. As the inclusion is an inner fibration, it is an embedding of a full subcategory. The vertex of $q^{-1}(\mathcal{C}')$ is precisely $\{b\in\mathcal{B}~|~q(b)\simeq 0\}$, and since there should be a unique full simplicial subset (in this case, automatically a full subcategory) of $\mathcal{B}$ with the given vertex set \cite[tag 01CV]{kerodon}, this concludes the proof. 
\end{proof}
\end{proposition}

In particular, for the case of exact functors $\mathcal{B}\xrightarrow{q}\mathcal{C}$ between stable $\infty$-categories, we call $\fib(q)$ (for any choice of a zero object) a \emph{fiber}, or even a \emph{kernel} of $q$. Following the description of Proposition \ref{prop:infinityfiber}, we identify $\fib(q)$ with the stable subcategory of $\mathcal{B}$ generated by $\{b\in\mathcal{B}~|~q(b)\simeq 0\}$. Note that this description is independent of the choice of zero objects or choice of isomorphisms between zero objects, which is not immediate from the definition of $\fib(q)$ as a pullback $\infty$-category.  \\
\indent An important class of exact sequences is provided by semiorthogonal decompositions of stable $\infty$-categories. Given a stable $\infty$-category $\mathcal{B}$ and its stable subcategory $\mathcal{C}$, we denote $\mathcal{C}^{\perp}$ as the full (stable) subcategory of $\mathcal{B}$ generated by $b\in\mathcal{B}$ with $\Map_{\mathcal{B}}(c,b)\simeq *$ for all $c\in\mathcal{C}$, and similarly denote ${}^{\perp}\mathcal{C}$ as the stable subcategory of $\mathcal{B}$ generated by $b\in\mathcal{B}$ with $\Map_{\mathcal{B}}(b,c)\simeq *$ for all $c\in\mathcal{C}$. 

\begin{proposition}\label{prop:splitexact}
Let $\mathcal{A}\to\mathcal{B}\xrightarrow{q}\mathcal{C}$ be a sequence in $\Prl_{\st}$. If $q$ admits a fully faithful right adjoint and $\mathcal{A}\to\mathcal{B}$ induces an equivalence between $\mathcal{A}$ and $\fib(q)$, then the sequence is exact in $\Prl_{\st}$. 
\begin{proof}
The only nontrivial part to check is that $q$ induces $\mathcal{B}/\mathcal{A}\simeq\mathcal{C}$. By assumption, we can identify the fully faithful embedding $\mathcal{A}\to\mathcal{B}$ as $\mathcal{A} = \fib(q)\hookrightarrow\mathcal{B}$. By \cite[5.6]{bgt}, cofiber $\mathcal{B}/\mathcal{A}$ is equivalent to the Bousfield localization of $\mathcal{B}$ at morphisms whose cofibers are in the essential image of $\mathcal{A}$, i.e., $\mathcal{B}/\mathcal{A}$ is equivalent to the stable subcategory of $\mathcal{B}$ generated by objects $b\in\mathcal{B}$ such that $\Map_{\mathcal{B}}(a,b)\simeq *$ for all $a\in\mathcal{A}$. If we denote a right adjoint of $\mathcal{A}\hookrightarrow\mathcal{B}$ by $g'$, then (from $\Map_{\mathcal{B}}(a,b)\simeq\Map_{\mathcal{A}}(a,g'(b))$ for all $a\in\mathcal{A}$) we have $\mathcal{B}/\mathcal{A}\simeq\fib(g')$. Now, via the fully faithful right adjoint $\mathcal{C}\hookrightarrow\mathcal{B}$ of $q$, let us identify $\mathcal{C}$ as a stable subcategory of $\mathcal{B}$. Then by adjunction $\mathcal{A} = \fib(q) = {}^{\perp}\mathcal{C}$. On the other hand, again using the right adjoint $g'$ one immediately computes $({}^{\perp}\mathcal{C})^{\perp} = \fib(g')$. From $\mathcal{C}$ = $({}^{\perp}\mathcal{C})^{\perp}$ \cite[7.2.1.8]{sag}, one has $\mathcal{C} = \fib(g')\simeq\mathcal{B}/\mathcal{A}$. 
\end{proof}
\end{proposition}

We call exact sequences of $\Prl_{\st}$ satisfying the conditions of Propostion \ref{prop:splitexact} \emph{split-exact}. The point of Proposition \ref{prop:splitexact} is that the conditions of \cite[5.18]{bgt} (in $\Cat^{\ex(\kappa)}$, with \emph{a priori} given exactness assumption) automatically ensure the sequence is exact in the case of $\Prl_{\st}$ (for instance, see \cite[Recollection 9]{tamme} for the statement). By our discussions on fiber and cofiber $\infty$-categories, a split-exact sequence $\mathcal{A}\to\mathcal{B}\xrightarrow{q}\mathcal{C}$ in $\Prl_{\st}$ is simultaneously a fiber and a cofiber sequence in $\Prl_{\st}$. Also, note that given that the stable subcategory $\mathcal{C}\hookrightarrow\mathcal{B}$ (via right adjoint) is closed under equivalences, the condition of Proposition \ref{prop:splitexact} is precisely saying we have a semiorthogonal decomposition of $\mathcal{B}$ of the form $(\mathcal{A},\mathcal{C})$ \cite[7.2.1.7]{sag}.  

\begin{example}\label{ex:splitexact}
Let $\mathcal{B}$ be a small stable $\infty$-category and let $\mathcal{A}$ be its stable subcategory. By \cite[I.3.5]{ns}, the resulting sequence $\Ind(\mathcal{A})\to\Ind(\mathcal{B})\to\Ind(\mathcal{B}/\mathcal{A})$ in $\Prl_{\st}$ exhibits $\Ind(\mathcal{A})$ as a fiber of the second compact functor $\Ind(\mathcal{B})\to\Ind(\mathcal{B}/\mathcal{A})$, and this functor admits a fully faithful right adjoint. Hence by Proposition \ref{prop:splitexact}, the sequence is a split-exact sequence in $\Prl_{\st}$. Moreover, the right adjoint $\Ind(\mathcal{B}/\mathcal{C})\to\Ind(\mathcal{B})$ is also in $\Prl_{\st}$ (i.e., preserves all small colimits), and corresponds to the Yoneda functor $\mathcal{B}/\mathcal{A}\to\Ind(\mathcal{B})$ sending the image of $b\in\mathcal{B}$ in $\mathcal{B}/\mathcal{A}$ to the filtered colimit $\colim_{a\in\mathcal{A}_{/b}}\cof(a\to b)$ in $\Ind(\mathcal{B})$ (i.e., $\colim_{a\in\mathcal{A}_{/b}}\Map_{\mathcal{B}}(-,\cof(a\to b))\in\mathcal{P}(\mathcal{B})$). This immediately follows from the description of the mapping space of $\mathcal{B}/\mathcal{A}$, as well as from the fact that $\Ind(\mathcal{B}/\mathcal{A})\to\Ind(\mathcal{B})$ is already exact, so it suffices to consider the case of filtered colimits when verifying that the functor commutes with all small colimits. In particular, the unit map for the adjunction associated with the second compact functor on $b\in\mathcal{B}$ takes the form $b\to\colim_{a\in\mathcal{A}_{/b}}\cof(a\to b)$.
\end{example}

\begin{remarkn}\label{rmk:splitexact}
Suppose we are given a split exact sequence $\mathcal{K}\to\mathcal{B}\xrightarrow{\jmath\pb}\mathcal{C}$ of $\Prl_{\st}$, with $\mathcal{K}$ given as a stable subcategory of $\mathcal{B}$ and $\jmath\pb:\mathcal{B}\to\mathcal{C}$ be in $\Prl_{\st,\omega}$ (i.e., a compact functor). Let $\imath\ush:\mathcal{B}\to\mathcal{K}$ and $\jmath\psh:\mathcal{C}\to\mathcal{B}$ be right adjoints of the functors $\mathcal{K}\hookrightarrow\mathcal{B}$ and $\jmath\pb:\mathcal{B}\to\mathcal{C}$ respectively.\\
(1) We have fiber sequences $\imath\ush b\to b\to \jmath\psh\jmath\pb b$ in $\mathcal{B}$ fuctorial on $b\in\mathcal{B}$ \cite[7.2.0.2]{sag}. Also note that $\imath\ush$ commutes with filtered colimits, as it is equivalent to a fiber of the unit map $id\to\jmath\psh\jmath\pb$, whose source and target functors commute with filtered colimits due to the compactness assumption on $\jmath\pb$.\\
(2) Suppose we have a stable subcategory $\mathcal{A}\subseteq\mathcal{K}$ closed under filtered colimits and suppose $\imath\ush$ maps compact objects $\mathcal{B}^{\omega}$ into $\mathcal{A}$. Then, $\imath\ush$ induces an equivalence $\mathcal{K}\to\mathcal{A}$, with an inverse given by inclusion $\mathcal{A}\subseteq\mathcal{K}$. In fact, by assumption $\imath\ush$ induces $\mathcal{B}\simeq \Ind(\mathcal{B}^{\omega})\to\mathcal{A}$, and satisfies $a\simeq \imath\ush a$ for $a\in\mathcal{A}$ by restriction of the unit map (which is an equivalence) on $\mathcal{A}$. For $k\in\mathcal{K} = \fib(\jmath\pb)$, the canonical fiber sequence $\imath\ush k\to k\to\jmath\psh\jmath\pb k$ has zero cofiber part, and hence the counit also induces an equivalence $\imath\ush k\simeq k$ for $k\in\mathcal{K}$.   
\end{remarkn}

We finally recall the notion of localizing invariants in stable setting. A functor $E:\Cat^{\ex}\to\mathcal{T}$ defined on small stable $\infty$-categories and valued in a stable $\infty$-category $\mathcal{T}$ is called \emph{localizing} if it factors through the idempotent-completion $\Cat^{\ex}\to\Cat^{\perf}$ and sends exact sequences of $\Cat^{\ex}$ to fiber sequences of $\mathcal{T}$. Archetypical examples are the nonconnective K-theory functor $K:\Cat^{\ex}\to\Sp$ of Bass-Thomason and various functors related to it via trace maps, e.g., $\operatorname{THH}$ and $\operatorname{TC}$. These are all valued in the $\infty$-category $\Sp$ of spectra, although $\mathcal{T}$ might be any stable $\infty$-category in theory---for instance $\operatorname{THH}$ is canonically valued in $\mathcal{T} = \operatorname{CycSp}$ and is a localizing invariant valued in $\mathcal{T}$. A localizing invariant $E$ is called \emph{finitary} if it commutes with filtered colimits---algebraic K-theory functor $K$ and $\operatorname{THH}$ are standard examples of finitary localizing invariants, while $\operatorname{TC}$ is not finitary. An important characterization of the K-theory functor is given by the corepresentability result of \cite[9.8]{bgt}. There is a finitary localizing invariant $[-]_{\loc}:\Cat^{\ex}\to\mathcal{M}_{\loc}$ into some $\mathcal{M}_{\loc}\in\Prl_{\st}$ which is universal, that any finitary localizing invariants into any $\mathcal{T}\in\Prl_{\st}$ uniquely (up to homotopy) factor through $[-]_{\loc}$, i.e., there is an equivalence $\Fun_{\loc}(\Cat^{\ex},\mathcal{T})\xleftarrow{-\circ[-]_{\loc}}\Fun^{\text{L}}(\mathcal{M}_{\loc},\mathcal{T})$. Via this equivalence, $K$ is described as $K(-)\simeq \map_{\mathcal{M}_{\loc}}([\Perf(\mathbb{S})]_{\loc},[-]_{\loc})$. More generally, mapping spectra $\map_{\mathcal{M}_{\loc}}([\mathcal{C}]_{\loc},[\mathcal{D}]_{\loc})$ in $\mathcal{M}_{\loc}$ from a smooth proper $\mathcal{C}\in\Cat^{\ex}$ (e.g., $\Perf(\mathbb{S})$) can be expressed as a K-theory spectrum (see \cite[9.36]{bgt} and for the additive version see \cite[9.9]{bgt}), and this often enables one to extend results about K-theory to results for localizing invariants in general. Due to its importance, we chose our title to refer to a descent result for the algebraic K-theory functor, although the result holds more generally for any localizing invariants. 

\subsection{Limits of diagrams}
\label{subsec:limit}

First, recall the following basic behaviours of cubical limits. 

\begin{proposition}\cite[1.2.4.13]{ha}
\label{prop:lim=colim}
Let $\mathcal{C}$ be a stable $\infty$-category, and let $n\geq 0$. Suppose we are given a diagram $F\in\Fun((\Delta^{1})^{n},\mathcal{C})$. Then, $F$ is a limit diagram iff $F$ is a colimit diagram. 
\end{proposition}

\begin{proposition}
\label{prop:fiberlimcrit}
Let $\mathcal{C}$ be a stable $\infty$-category, and let $n\geq1$. Suppose we are given a diagram $F\in\Fun((\Delta^{1})^{n},\mathcal{C})$, which can be identified with an object $F'$ of $\Fun\left(\Delta^{1},\mathcal{C}^{(\Delta^{1})^{n-1}}\right)$ by choosing a component $\Delta^{1}$ of $(\Delta^{1})^{n}$. Take any choice of a fiber functor $\Fun\left(\Delta^{1},\mathcal{C}^{(\Delta^{1})^{n-1}}\right)\xrightarrow{\fib_{n-1}}\Fun((\Delta^{1})^{n-1},\mathcal{C})$. Then, $F$ is a limit diagram if and only if $\fib_{n-1}(F')$ is a limit diagram. 
\begin{proof}
This follows immediately from \cite[1.2.4.15]{ha}. More precisely, it treats the (more general) case of colimits over any simplicial set $K$ whose shapes in $\mathcal{C}$ admit colimits. In our case, $K$ can be taken as the finite simplicial set satisfying $K^{\lhd} = (\Delta^{1})^{n-1}$. Combined with Proposition \ref{prop:lim=colim} above, we have the result. 
\end{proof}
\end{proposition}

Let us investigate a relationship between cubical and semi-cosimplicial limits. View $\mathcal{P}(\mathbb{N})$ as a small category via its poset structure determined by inclusions of subsets. Also, recall that the standard semi-simplicial category $\Delta_{s}$ is the subcategory of the standard simplicial category $\Delta$ with the same objects but only with injective order-preserving maps as morphisms (in other words, degeneracy maps are dropped), cf. \cite[6.5.3.6]{htt}. Let $\mathcal{P}(\bbN)\backslash\emptyset\xrightarrow{c}\Delta_{s}$ be the functor determined by sending $T = (0\leq i_{0}<\cdots<i_{r})\subseteq\bbN$ to $[r]$, and $T\backslash i_{k}\hookrightarrow T$ to the $k$-th face map $[r-1]\xrightarrow{d^{k}}[r]$. For each $n\geq 0$, it restricts to functors $\mathcal{P}([n])\backslash\emptyset\xrightarrow{c_{n}}(\Delta_{s})_{\leq n}$. 

\begin{lemma}\label{lem:cartfib}
The functor $\mathcal{P}(\bbN)\backslash\emptyset\xrightarrow{c}\Delta_{s}$, and hence $\mathcal{P}([n])\backslash\emptyset\xrightarrow{c_{n}}(\Delta_{s})_{\leq n}$ for each $n\geq0$, is a Cartesian fibration of ordinary categories. 
\begin{proof}
Suppose we are given $\emptyset\neq T = (0\leq i_{0}<\cdots<i_{r'})\subseteq\bbN$ and $[r]\xrightarrow{\alpha}[r']$ in $\Delta_{s}$. We have to check that there is a $c$-Cartesian lifting of $\alpha$ in $\mathcal{P}(\mathbb{N})\backslash\emptyset$ whose target is $T$ \cite[tag 01RN]{kerodon}. As $\alpha$ is injective, we have a well-defined $S:=(0\leq i_{\alpha(0)}<\cdots<i_{\alpha(r)})\subseteq T$ such that $c(S) = [r]$. Note that any inclusion $S' = (0\leq j_{0}<\cdots<j_{s})\subseteq T$ with $c(S') = [s]$ maps to $[s]\xrightarrow{\alpha'}[r']$ in $\Delta_{s}$ in a way that $j_{k} = i_{\alpha'(k)}$. Thus, the morphism $S\subseteq T$ lifts $\alpha$. To check $S\subseteq T$ is $c$-Cartesian, suppose we are given $S'\subseteq T$ as before and a morphism $[s]\xrightarrow{\beta}[r]$ in $\Delta_{s}$ such that $S'\subseteq T$ maps to $\alpha\circ\beta$ by $c$. Then, $j_{k} = i_{\alpha(\beta(k))}\in S$ for each $k\in[s]$, and hence we know $S'\subseteq S$. The image $[s]\xrightarrow{\gamma}[r]$ of the morphism $S'\subseteq S$ satisfies $j_{k} = i_{\alpha(\gamma(k))}$ for each $k\in[s]$. As $\alpha$ is injective, we know $\gamma = \beta$, i.e., $S'\subseteq S$ lifts $\beta$. 
\end{proof}
\end{lemma}

We explain a slight generalization of \cite[Proposition 40.2]{cs} in an $\infty$-categorical setting for our purpose. Loosely speaking, this interprets an intergration along the fibers formula for (co)limits over Grothendieck constructions. Given an $\infty$-category $\mathcal{C}$, let us denote the Grothendieck construction by $\Fun(\mathcal{C},\Cat_{\infty})\underset{~}{\xrightarrow{\Gr}}\co\CFib(\mathcal{C})$, and the dual construction by $\Fun(\mathcal{C}^{\op},\Cat_{\infty})\underset{~}{\xrightarrow{\Gr^{-}}}\CFib(\mathcal{C})$. (Here, $\co\CFib(\mathcal{C})$ and $\CFib(\mathcal{C})$ denotes the $\infty$-category of coCartesian fibrations and Cartesian fibrations over $\mathcal{C}$ respectively.) 

\begin{proposition}\label{prop:fiberlim}
Let $H\in\Fun(\mathcal{C},\Cat_{\infty})$ be a functor from an $\infty$-category $\mathcal{C}$, and let $F\in\Fun(\Gr(H),\mathcal{E})$ be a functor from $\Gr(H)$ into an $\infty$-category $\mathcal{E}$. Suppose $\mathcal{E}$ admits colimits indexed over $H(c)$ for each $c\in\mathcal{C}$, as well as over $\mathcal{C}$. Then, a colimit of $F$ exists\footnote{Note that by the proof below, we can weaken the existence of colimits condition slightly: assume $\mathcal{E}$ admits $H(c)$-indexed colimits for all $c$. Then, it suffices to require the LHS colimit $\colim_{\mathcal{C}} p\lsh F$ of the formula exists, rather than all colimits over $\mathcal{C}$.}. Moreover, there exists a functor $p\lsh F\in\Fun(\mathcal{C},\mathcal{E})$ such that each $p\lsh F(c)$ is equivalent to a colimit of $F|_{H(c)}$, and there is an equivalence between colimits $\colim_{\mathcal{C}} p\lsh F\simeq \colim_{\Gr(H)} F$ in $\mathcal{E}$ canonical on $F$. 
\begin{proof}
Let $\Gr(H)\xrightarrow{p}\mathcal{C}$ be a coCartesian fibration corresponding to $H$, and let $\Fun(\mathcal{C},\mathcal{E})\xrightarrow{p\pb}\Fun(\Gr(H),\mathcal{E})$ be the induced functor. By assumption its left adjoint, the functor of left Kan extensions $\Fun(\Gr(H),\mathcal{E})\xrightarrow{p\lsh}\Fun(\mathcal{C},\mathcal{E})$ exists, and is computed pointwisely. More precisely, by \cite[1.16]{mg}, each $H(c)$ is canonically equivalent to $\Gr(H)\times_{\mathcal{C}}c$, and this fiber product as $\infty$-categories is equivalent to the fiber product computed as quasicategories (simplicial sets). Thus \cite[4.3.3.10]{htt} (with $q = id_{S}$ and $\delta = p$) applies to ensure $p\lsh$ exists, and satisfies $(p\lsh F)(c)\simeq \colim(H(c)\hookrightarrow \Gr(H)\xrightarrow{F}\mathcal{E})$. Again by assumption a left Kan extension $s\lsh p\lsh F\simeq\colim_{\mathcal{C}}p\lsh F$ of $p\lsh F$ along $\mathcal{C}\xrightarrow{s}\Delta^{0}$ exists, and it gives a left Kan extension of $F$ along $s\circ p:\Gr(H)\to\Delta^{0}$.     
\end{proof}
\end{proposition}

\begin{remarkn}
(1) Informally speaking, Proposition \ref{prop:fiberlim} says that there is a canonical equivalence 
\begin{align*}
\colim_{c\in\mathcal{C}}\left(\colim\left(H(c)\hookrightarrow\Gr(H)\xrightarrow{F}\mathcal{E}\right)\right)\simeq \colim_{\Gr(H)}F. 
\end{align*}  
(2) Dually, given $H\in\Fun(\mathcal{C}^{\op},\Cat_{\infty})$ and $F\in\Fun(\Gr^{-}(H),\mathcal{E})$ such that $\mathcal{E}$ admits limits indexed over $H(c)$ (for all $c\in\mathcal{C}$) and $\mathcal{C}$, we know a limit of $F$ exists, and have a canonical equivalence
\begin{align*}
\lim_{\Gr^{-}(H)}F\simeq \lim_{c\in\mathcal{C}}\left(\lim\left(H(c)\hookrightarrow\Gr^{-}(H)\xrightarrow{F}\mathcal{E}\right)\right).
\end{align*}
More precisely, we have an equivalence between limits $\lim_{\Gr^{-}(H)}F\simeq \lim_{\mathcal{C}}p\psh F$ in $\mathcal{E}$, and the functor $p\psh F$ is given as a right Kan extension of $F$ along the Cartesian fibration $\Gr^{-}(H)\xrightarrow{p}\mathcal{C}$ corresponding to $H$. 
\end{remarkn}

\begin{corollary}\label{cor:fiberlim}
Let $F:(\Delta^{1})^{n+1}\backslash\emptyset\simeq \N\mathcal{P}([n])\backslash\emptyset\to\mathcal{T}$ be a $n$-cubical diagram (without the initial vertex) valued in a finitely complete $\infty$-category $\mathcal{T}$. Then, its limit $\lim_{(\Delta^{1})^{n+1}\backslash\emptyset}F$ exists, and is equivalent to $\lim_{[r]\in(\Delta_{s})_{\leq n}}\left(\prod_{0\leq i_{0}<\cdots<i_{r}\leq n}F(i_{0},...,i_{r})\right)$. 
\begin{proof}
Consider the Cartesian fibration $\mathcal{P}([n])\backslash\emptyset\xrightarrow{c_{n}}(\Delta_{s})_{\leq n}$ of Lemma \ref{lem:cartfib}. Note that each fiber category $c_{n}\inv([r])$ is finite discrete. Also, note that the nerve $\N(\Delta_{s})_{\leq n}$ viewed as a simplicial set is finite. Indeed if $i>n$, then each $i$-simplex in $\Fun([i],(\Delta_{s})_{\leq n})$ viewed as a composition of $i$-number of morphisms must contain an identity morphism. For any ordinary category $\mathcal{C}$, an $i$-simplex of $\N\mathcal{C}$ is nondegenerate precisely if it can be represented by a sequence $x_{0}\to\cdots\to x_{i}$ of morphisms which does not include identities, so nondegenerate simplices of $\N(\Delta_{s})_{\leq n}$ are concentrated in degrees $\leq n$, and hence their number is finite. By Proposition \ref{prop:fiberlim}, a limit $\lim_{(\Delta^{1})^{n+1}\backslash\emptyset}F$ exists, and is equivalent to $\lim((c_{n})\psh F)\simeq \lim_{[r]\in(\Delta_{s})_{\leq n}}\left(\prod_{0\leq i_{0}<\cdots<i_{r}\leq n}F(i_{0},...,i_{r})\right)$. 
\end{proof}
\end{corollary}

\subsection{Sheaves of modules}
\label{subsec:modules}

We briefly recall some conventions and results of \cite{sag} about sheaves of modules which we will use here. Informally speaking, we consider modules over Noetherian (ordinary) schemes in a derived sense. In particular, functors between module categories should be read as derived ones of their classical counterparts unless otherwise specified (e.g., in construction of adele rings). More precisely, given a qcqc (ordinary) scheme\footnote{Quasicompact quasiseparated schemes. For example, affine schemes and Noetherian schemes are qcqs.} $X$, we consider the symmetric monoidal $\infty$-category $\Mod(\mathscr{O}_{X})\in\Prl_{\st}$ of sheaves of $\mathscr{O}_{X}$-module spectra, and likewise $\Mod(\mathscr{O})$ for any Zariski sheaf $\mathscr{O}$ of discrete commutative rings on $X$ \cite[2.1.0.1]{sag}. The $\infty$-category $\Mod(\mathscr{O})$ has a canonical t-structure specified via connective $\mathscr{O}$-modules, and its heart $\Mod(\mathscr{O})^{\heartsuit}$ recovers the abelian category of discrete $\mathscr{O}$-modules.
\begin{remarkn}\label{rmk:discretemod}
The canonical functor $\mathcal{D}(\Mod(\mathscr{O})^{\heartsuit})\to\Mod(\mathscr{O})$ from the derived $\infty$-category induced from the canonical embedding of the heart is a fully faithful embedding, and identifies $\mathcal{D}(\mathscr{O}) = \mathcal{D}(\Mod(\mathscr{O})^{\heartsuit})$ with the stable subcategory spanned by $\mathscr{O}$-modules whose underlying sheaves of spectra are hypercomplete \cite[2.1.2.3]{sag}. In particular if the underlying $\infty$-topos $\Shv_{\mathcal{S}}(X_{\text{Zar}})$ associated with $X$ is hypercomplete, then it is an equivalence. For example, the assumption holds if $X$ is Noetherian of finite Krull dimension. 
\end{remarkn}

The theory of quasicoherent sheaves \cite[2.2]{sag} applied to $X = (X,\mathscr{O}_{X})$ gives $\QCoh(X)\in\Prl_{\st}$ as a stable subcategory of $\Mod(\mathscr{O}_{X})$, which inherits a symmetric monoidal structure and a t-structure recovering the abelian category of discrete quasicoherent sheaves on $X$ as $\QCoh(X)^{\heartsuit}\subseteq\Mod(\mathscr{O}_{X})^{\heartsuit}$. For $X = \Spec R$ affine we recover $\QCoh(\Spec R)\simeq \Mod(HR)$, which we simply write as $\Mod(R)$ \cite[2.2.3.3]{sag}. 
\begin{remarkn}
For a qcqs scheme $X$, let $\mathcal{D}_{\QCoh}(X) = \mathcal{D}_{\QCoh}(\Mod(\mathscr{O}_{X})^{\heartsuit})$ be the stable subcategory of $\mathcal{D}(\mathscr{O}_{X}) = \mathcal{D}(\Mod(\mathscr{O}_{X})^{\heartsuit})$ spanned by $\mathscr{O}_{X}$-modules whose homologies are discrete quasicoherent sheaves. Then, the embedding of Remark \ref{rmk:discretemod} induces an equivalence $\mathcal{D}_{\QCoh}(X)\simeq\QCoh(X)$ \cite[2.2.6.2]{sag}. 
\end{remarkn}

Let $R$ be a discrete commutative ring\footnote{Or in fact any connective $\mathbb{E}_{\infty}$-ring.}. Recall that the $\infty$-category of perfect $R$-modules is the smallest stable subcategory $\Perf(R)$ of $\Mod(R)$ containing $R$ and closed under retractions. In fact, $\Perf(R) = \Mod(R)^{\text{dual}} = \Mod(R)^{\omega}$, i.e., perfect modules are precisely dualizable modules, and are again precisely compact modules \cite[7.2.4]{ha}. For qcqs schemes such identification remains true. Let $X$ be a qcqs scheme, and let $\Perf(X)$ be the stable subcategory of $\QCoh(X)$ spanned by quasicoherent modules affine-locally perfect. Equivalently, $\Perf(X)\simeq\lim_{\Spec R\to X}\Perf(R)$; see also \cite[tag 08CM]{stacks} for a classical approach. Then $\Perf(X) = \QCoh(X)^{\text{dual}} = \QCoh(X)^{\omega}$ \cite[6.2.6.2, 9.1.5.5]{sag}, where $\QCoh(X)^{\text{dual}}$ is the full subcategory spanned by dualizable objects in the symmetric monoidal $\infty$-category $\QCoh(X)$. Note that the compactness assumption (qcqs property) is needed precisely for the identification of compactness and dualizability. We in particular know the functor $X\mapsto \QCoh(X)$ on qcqs schemes is valued in $\Prl_{\st,\omega}$.\\
\indent Let $E:\Cat^{\ex}\to\mathcal{T}$ be a localizing invariant valued in a stable $\infty$-category $\mathcal{T}$. Composing with $\Perf(-):\Sch_{\text{qcqs}}^{\op}\to\Cat^{\perf}\subseteq\Cat^{\ex}$ defined via pullbacks, we can view $E$ as a functor defined on (the nerve of) the category of qcqs schemes, and we set $E(X) := E(\Perf(X))$. After restriction, it is in particular defined on the category of discrete commutative rings, and we denote $E(R) := E(\Perf(R))$. As $E$ is localizing, it is an additive invariant \cite[6.1]{bgt} and in particular $E$ commutes with finite products of rings. Note that localizing (or additive) invariants however do not commute with finite limits, even pullbacks of rings in general. 
\begin{remarkn}\label{rmk:variousperfect}
(1) For product $\prod_{i\in I}R_{i}$ of commutative rings (indexed by a small set $I$ which might not be finite), Bhatt's theorem \cite{bhatt} guarantees the map $\Perf(\prod_{i\in I}R_{i})\to\prod_{i\in I}\Perf(R_{i})$ induced by projections of rings is a fully faithful embedding, cf. \cite[Theorem 3.15]{adelic}.\\
(2) On (almost-)perfect modules over Noetherian rings, extension of scalars by completions realize derived completions, see \cite[7.2 and 7.3]{sag}. Let $R$ be a Noetherian commutative ring and $I$ be its ideal. For $C\in\Perf(R)$ the canonical map $R^{\wedge}_{I}\otimes_{R}C\to C^{\wedge}_{I}$ is an equivalence in $\Mod(R)$ \cite[7.3.5.7]{sag}. In particular $R^{\wedge}_{I}\otimes_{R}C\in\Mod^{\Cpl(I)}(R)$ is an $I$-complete object.\\
(3) There is a version of derived Nakayama lemma for $I$-complete modules. Let $R$ be a (discrete) Noetherian commutative ring and let $I$ be an ideal of $R$. Let $C\in\Mod^{\Cpl(I)}(R)$ be an $I$-complete module. If $C$ in $\Mod(R)$ satisfies $R/I\otimes_{R}C\simeq 0$, then $C\simeq 0$ \cite[0G1U]{stacks}. 
\end{remarkn}

\subsection{Adeles on Noetherian schemes}
\label{subsec:adeles}

Let $X$ be a Noetherian scheme. Its underlying set of points admits a canonical partial order given by specializations of points, i.e., for points $p$ and $q$ of $X$, we say $p\leq q$ if $p\in\overline{q}$ (i.e., if $p$ is a specialization of $q$). We write the simplicial set obtained as the nerve of the poset structure on $X$ as $S_{\cdot}(X) = \N(X)$. By definition for each $r\geq 0$, one has $S_{r}(X) = \{(p_{0},...,p_{r})\in X^{r+1}~|~p_{0}\leq p_{1}\leq\cdots\leq p_{r}\}$. We also consider the semi-simplicial set $S_{\cdot}^{\text{red}}(X)$ consisting of $S_{r}^{\text{red}}(X) = \{(p_{0},...,p_{r})\in X^{r+1}~|~p_{0}< p_{1}<\cdots<p_{r}\}$ for each $r\geq 0$. After restriction to $\Delta_{s}^{\op}$ we can view $S_{\cdot}(X)$ as a semi-simplicial set, and $S^{\text{red}}_{\cdot}(X)$ is defined to be its semi-simplicial subset with vertices specified as above. \\
\indent For each $r\geq 0$, a subset $T\subseteq S_{r}(X)$, and a quasicoherent sheaf $F\in\QCoh(X)^{\heartsuit}$, we can define the sheaf of adeles $A_{T}(F) = A(T,F)$ as an object of $\Mod(\mathscr{O}_{X})^{\heartsuit}$. Below, $T_{q} = \{(p_{0},...,p_{r})\in T~|~p_{r} = q\}$ and $h_{sq}:\Spec\mathscr{O}_{q}/\mathfrak{m}_{q}^{r}\to X$ is the canonical map for each $q\in X$ and $s\geq 0$. Functors $(h_{sq})\psh$ and $h_{sq}\pb$ in this subsection are underived pushforwards and pullbacks respectively. 
\begin{definition}
For each $r\geq 0$ and $T\subseteq S_{r}(X)$, we let $A_{T} = A(T,-):\QCoh(X)^{\heartsuit}\to\Mod(\mathscr{O}_{X})^{\heartsuit}$ be the exact functor uniquely characterized by the following three conditons \cite[Proposition 2.1.1]{huber}, \cite[Definition 1.4]{adelic}: \\
(1) $A_{T}$ commutes with filtered colimits\footnote{Hence it suffices to determine values of $A_{T}$ on each coherent sheaves on $X$.}. \\
(2) If $r = 0$, then $A_{T}(F) = \prod_{q\in T}\lim_{s\geq 0}(h_{sq})\psh h_{sq}\pb F$ for each coherent sheaf $F$. \\
(3) If $r>0$, then $A_{T}(F) = \prod_{q\in X}\lim_{s\geq 0} A_{T_{q}}((h_{sq})\psh h_{sq}\pb F)$ for each coherent sheaf $F$.
\end{definition}

\noindent By taking local sections, we recover abelian groups of adeles associated with $T$ and $F$ restricted on each opens. Also, by construction each functor $A_{T}$ is lax symmetric monoidal, so each $A_{T}(\mathscr{O}_{X})$ is canonically a sheaf of commutative $\mathscr{O}_{X}$-algebras. We will often omit $\mathscr{O} = \mathscr{O}_{X}$ in the notation, and simply write as $A_{T} = A_{T}(\mathscr{O}) = A(T,\mathscr{O})$. Note that over an affine $X = \Spec R$, global sections rings $\Gamma(A_{T})$ are flat over $R$ due to exactness of $A_{T}$ \cite[Lemma 1.10]{adelic}.\\
\indent It turns out that the construction of $A_{T}(F)$ is also sufficiently functorial on $T$. In fact, for each $F\in\QCoh(X)^{\heartsuit}$ the association $[r]\mapsto A^{r}(X,F):= A(S_{r}(X),F)$ assembles to a cosimplicial object $A^{\sbullet}(X,F)$ of $\Mod(\mathscr{O}_{X})^{\heartsuit}$, and likewise the association $[r]\mapsto A^{r}_{\text{red}}(X,F):= A(S^{\text{red}}_{r}(X),F)$ assembles to a semi-cosimplicial object $A^{\sbullet}_{\text{red}}(X,F)$ of $\Mod(\mathscr{O}_{X})^{\heartsuit}$ \cite[Theorem 2.4.1]{huber}, \cite[Theorem 8.12]{intro}. See also \cite[Proposition 1.7]{adelic}.

\begin{remarkn}\label{rmk:functoriality}
Let us briefly review the functoriality of $A_{T}(F)$ on $T$, and in particular explain how the semi-cosimplicial object $A^{\sbullet}_{\text{red}}(X,F)$ is defined. It suffices to describe maps between local sections, and after restriction we are reduced to the case of global sections. So let us abuse notations slightly and understand $A_{T}(F)$ as the module of global sections. By \cite[Proposition 2.1.4]{huber} for each $T\subseteq S_{r}^{\text{red}}(X)$, there is an embedding $A_{T}(F)\hookrightarrow \prod_{\xi\in T}A_{\xi}(F)$ into the product of local factors $A_{\xi}(F) = A(\{\xi\},F)$ canonical on $F$. Suppose we are given a map $\alpha:[r]\to [r']$ of $\Delta_{s}$, such that the induced $\alpha\pb: S_{r'}^{\text{red}}(X)\to S_{r}^{\text{red}}(X) = (p_{0},...,p_{r'})\mapsto (p_{\alpha(0)},...,p_{\alpha(r)})$ maps $S\subseteq S_{r'}^{\text{red}}(X)$ into $T\subseteq S_{r}^{\text{red}}(X)$. For each $\eta\in S^{\text{red}}_{r'}(X)$, there is a canonical map $\alpha_{\eta}:A_{\alpha\pb(\eta)}(F)\to A_{\eta}(F)$ of local factors \cite[8.3, p. 59]{intro}, \cite[Definition 2.2.3]{huber}. Now, one defines the map $\alpha_{*,F}:\prod_{\xi\in T}A_{\xi}(F)\to\prod_{\eta\in S}A_{\eta}(F)$ as the composition $\prod_{\xi\in T}A_{\xi}(F)\to\prod_{\eta\in S}A_{\alpha\pb(\eta)}(F)\xrightarrow{\prod_{\eta\in S}\alpha_{\eta}}\prod_{\eta\in S}A_{\eta}(F)$, where the first map is induced from canonical projections. By \cite[Proof of Theorem 2.4.1 and Proposition 2.2.4]{huber} applied to a decomposition of $\alpha$ into composition of face maps, we know $\alpha_{*,F}$ induces the map $\alpha_{*,F}:A_{T}(F)\to A_{S}(F)$, and satisfies transitivity $\alpha_{*,F}\circ\beta_{*,F} = (\alpha\circ\beta)_{*,F}$ for $\beta$ satisfying analogous conditions as $\alpha$. In particular, $\prod_{\xi\in S_{\cdot}^{\text{red}}(X)}A_{\xi}(F)$ and $A^{\sbullet}_{\text{red}}(X,F):= A(S^{\text{red}}_{\sbullet}(X),F)$ are well-defined as semi-cosimplicial objects. Moreover, for a quasicoherent sheaf $B$ of (discrete) commutative $\mathscr{O}_{X}$-algebras induced maps between local factors and adeles are maps of algebras, and both $\prod_{\xi\in S_{\cdot}^{\text{red}}(X)}A_{\xi}(B)$ and $A^{\sbullet}_{\text{red}}(X,B)$ are semi-cosimplicial objects in commutative $\mathscr{O}_{X}$-algebras.
\end{remarkn}

We will consider $\infty$-category of modules over sheaves $A_{T}$ of adele rings on $X$. Let $A_{T}$ be the sheaf of adele rings associated with $T\subseteq S_{r}^{\text{red}}(X)$ and $\mathscr{O}_{X}$, and let $\Perf(A_{T})$ be the stable subcategory of $\mathcal{D}(A_{T})\hookrightarrow \Mod(A_{T})$ spanned by perfect complexes over $A_{T}$ \cite[tag 08CM]{stacks}, which we simply call as perfect $A_{T}$-modules. These are objects of $\mathcal{D}(A_{T})$ which are Zariski-locally on $X$ equivalent to objects represented by bounded complexes of direct summands of finite free $A_{T}$-modules. 

\begin{remarkn}\label{rmk:perfglobalsection}
In fact, we can identify this category with the $\infty$-category of perfect modules over the global section ring. Let $A_{T}$ be the sheaf of adele rings as above, and let $\Gamma(A_{T})$ be its ring of global sections. By \cite[Corollary 2.23]{adelic}, the global sections functor induces an equivalence $\Perf(A_{T})\simeq \Perf(\Gamma(A_{T}))$ of (symmetric monoidal) stable $\infty$-categories. In particular, their values $E(A_{T}) := E(\Perf(A_{T}))$ and $E(\Gamma(A_{T})) = E(\Perf(\Gamma(A_{T})))$ for each localizing invariant $E$ are equivalent. 
\end{remarkn}
 
Let $X$ be a Noetherian scheme of finite Krull dimension $n$. Given an increasing sequence $0\leq i_{0}<\cdots <i_{r}\leq n$ of integers, let $\underline{i_{0},...,i_{r}}:=\{(p_{0},...,p_{r})\in S^{\text{red}}_{r}(X)~|~\dim\overline{p_{k}} = i_{k}~\text{for all}~0\leq k\leq r\}$. Note that $S_{r}^{\text{red}}(X) =\coprod_{0\leq i_{0}<\cdots<i_{r}\leq n}\underline{i_{0},...,i_{r}}$. Hence by \cite[Proposition 2.1.5]{huber} the sheaf of reduced adele ring $A^{r}_{\text{red}}(X) := A(S_{r}^{\text{red}}(X),\mathscr{O}_{X})$ decomposes into $A^{r}_{\text{red}}(X)\cong\prod_{0\leq i_{0}<\cdots<i_{r}\leq n}A(i_{0},...,i_{r})$, where $A(i_{0},...,i_{r}):=A(\underline{i_{0},...,i_{r}},\mathscr{O}_{X})$. As each subset $S\subseteq [n]$ defines a unique increasing sequence $0\leq i_{0}<\cdots < i_{r}\leq n$ consisting of its elements, we use the notation $A(S) = A(\underline{i_{0},...,i_{r}},\mathscr{O}_{X})$ for any such $S$\footnote{Hence for example $A(ji) = A(ij) = A(\underline{ij},\mathscr{O}_{X})$ for $j>i$. }.

\begin{example}
Let $X$ be a Noetherian scheme of dimension $1$. Then $A^{\cdot}(X)$ takes the form of $F\times O\rightrightarrows A\times (F\times O)~\substack{\rightarrow\\[-1em] \rightarrow \\[-1em] \rightarrow}\cdots$, and similarly $A^{\cdot}_{\text{red}}(X)$ is of the form $F\times O\rightrightarrows A$ \cite[Proposition. 3.3.3]{huber}. Here $F = A(1)$ is the sheaf of rings of fractions of $X$ whose global section ring is $\prod_{\eta\in X_{1}}\mathscr{O}_{\eta}$ (here the product is taken over generic points of $X$), $O = A(0)$ is the sheaf of integral adele rings of $X$ whose global section ring is $\prod_{p\in X_{0}}\mathscr{O}^{\wedge}_{p}$ (where the product is taken over closed points), and $A = A(01) = F\otimes_{\mathscr{O}}O$ is the sheaf of finite adele rings of $X$. Hence the classical notion of finite adeles for global fields fits into the framework of higher adeles.  
\end{example}

\begin{remarkn}\label{rmk:cubicaladeles}
Let $X$ be a Noetherian scheme of finite Krull dimension $n$. \\
(1) The association $A:=T\mapsto A(T):\mathcal{P}([n])\to\CAlg(\mathscr{O}_{X})^{\heartsuit}$ gives a cubical object. For $\emptyset\not= S = (0\leq i_{0}<\cdots <i_{r}\leq n)\subseteq [n]$, denote $\underline{S} = \underline{i_{0},...,i_{r}}\subseteq S^{\text{red}}_{c(S)}(X)$ and $A(S) = A(\underline{S},\mathscr{O}_{X})$ as above, where $c(S) = c_{n}(S) = [r] = [|S|-1]$. Each $\emptyset\not= S\subseteq T\subseteq [n]$ induces a map $c_{S\subseteq T} = c(\imath_{S\subseteq T}):c(S)\to c(T)$ such that $c_{S\subseteq T}\pb: S^{\text{red}}_{c(T)}(X)\to S^{\text{red}}_{c(S)}(X)$ satisfies $c_{S\subseteq T}\pb(\underline{T})\subseteq\underline{S}$, so from the transitivity of Remark \ref{rmk:functoriality} we know the association $(S\subseteq T)\mapsto (c_{S\subseteq T})_{*,\mathscr{O}_{X}}:A(S)\to A(T)$ defines a functor $\mathcal{P}([n])\backslash\emptyset\to\CAlg(\mathscr{O}_{X})^{\heartsuit}$. By defining $A(\emptyset) = \mathscr{O}_{X}$, we have an extension of the functor to the cube $\mathcal{P}([n])$. \\
(2) Let $\alpha:[r]\to [r']$ be a map in $(\Delta_{s})_{\leq n}$. For each $T\in\mathcal{P}([n])\backslash\emptyset$ of $c(T) = [r']$, let $\alpha\pb T$ be the element of $\mathcal{P}([n])\backslash\emptyset$ with $c(\alpha\pb T) = [r]$ obtained by the Cartesian fibration structure of $c_{n}$. Then the induced map $((c_{n})\psh A)[r]\to((c_{n})\psh A)[r']$ is described as a composition $\prod_{S\in c_{n}^{-1}([r])}A(S)\to\prod_{T\in c_{n}^{-1}([r'])}A(\alpha\pb T)\to \prod_{T\in c_{n}^{-1}([r'])}A(T)$. Here, the first map is induced from the projections $\prod_{S\in c_{n}^{-1}([r])}A(S)\to A(\alpha\pb T)$ for each $T\in c_{n}^{-1}([r'])$, and the second map is the product of the maps $A(\alpha\pb T)\to A(T)$ induced from $\alpha\pb T\subseteq T$ over $T\in c_{n}^{-1}([r'])$. Let $E:\Cat^{\ex}\to\mathcal{T}$ be a localizing invariant. As the functor $E$ on rings commutes with finite products and as $E$ does not distinguish sheaves of adeles and their global section rings, \cite[4.3.3.10]{htt} implies $E((c_{n})\psh A)$ is equivalent to a right Kan extension of $E(A) = E\circ A$ along $c_{n}$. On the other hand, on each local sections rings the induced map $\prod_{S\in c_{n}^{-1}([r])}A(S)\to \prod_{T\in c_{n}^{-1}([r'])}A(T)$ for each $\alpha:[r]\to [r']$ is compatible with the map $\alpha_{*,\mathscr{O}_{X}}:\prod_{\xi\in S^{\text{red}}_{r}}A_{\xi}(\mathscr{O}_{X})\to\prod_{\eta\in S^{\text{red}}_{r'}}A_{\eta}(\mathscr{O}_{X})$ by construction (through embedding into products of local factors). As the semi-cosimplicial object $\prod_{\xi\in S^{\text{red}}_{\cdot}}A_{\xi}(\mathscr{O}_{X})$ induces $A^{\sbullet}_{\text{red}}(X)$ by restriction of each structure maps, we know $(c_{n})\psh A\cong A^{\sbullet}_{\text{red}}(X)$. 
\end{remarkn}

\section{Descent result}
In this section we explain the proof of our main result, Theorem \ref{thm:adelicdescent}. In \ref{subsec:exactseq} we explain the construction of certain exact sequences in $\Cat^{\ex}$ involving categories of perfect modules over sheaves of adeles (Proposition \ref{prop:exactseqadele}). Using these exact sequences and properties of cubical and semisimplicial diagrams explained in previous sections, we derive the descent result in \ref{subsec:adelicdescent}. 

\subsection{Exact sequences of categories of modules over adele rings}
\label{subsec:exactseq}

Let $X$ be a Noetherian scheme of finite Krull dimension $n$. 

\begin{definition}\label{def:perf}
For each $0\leq i\leq n$, denote the stable subcategory of $\Perf(X)$ generated by $C\in\Perf(X)$ with $\Supp(C)\subseteq X\backslash \left(X_{n}\cup\cdots\cup X_{i+1}\right)$ by $\Perf_{\leq i}(X)$. Here, $X_{j} := \{p\in X~|~\dim \overline{p} = j\}$ and $\Perf_{\leq n}(X) = \Perf(X)$. Hence, $\Perf_{\leq i}(X)$ consists of $C\in\Perf(X)$ with $\dim \Supp(C)\leq i$. Similarly, let $0\leq i_{0}<\cdots <i_{r}\leq n$, and denote the stable subcategory of $\Perf(A(i_{0},...,i_{r}))$ generated by the essential image of the exact functor $\Perf_{\leq i}(X)\hookrightarrow \Perf(X)\xrightarrow{A(i_{0},...,i_{r})\otimes_{\mathscr{O}}-}\Perf(A(i_{0},...,i_{r}))$ by $\Perf_{\leq i}(A(i_{0},...,i_{r}))\in\Cat^{\ex}$. 
\end{definition}  

\begin{remarkn}\label{rmk:idem}
Let $A = A(i_{0},...,i_{r})$ as in the definition above. The stable subcategory $\Perf_{\leq n}(A)$ of $\Perf(A)$ may not be closed under retracts, but is closed under finite colimits and contains $A$. Thus, after idempotent completion $\text{Idem}(\Perf_{\leq n}(A))\simeq \Ind(\Perf_{\leq n}(A))^{\omega}\simeq \Mod(A)^{\omega}\simeq \Perf(A)$ in $\Cat^{\perf}$, and localizing invariants do not distinguish between $\Perf_{\leq n}(A)$ and $\Perf(A)$.  \\
\indent On the other hand, by decreasing induction on $i\leq n$ one observes $\Perf_{\leq i}(X)\in\Cat^{\perf}$. For $i=n$ one has $\Perf_{\leq n}(X) = \Perf(X)\in\Cat^{\perf}$. In general one has a fiber sequence $\Perf_{\leq i}(X)\to\Perf_{\leq i+1}(X)\to \Perf(A(i+1))$, since for $C\in\Perf_{\leq i+1}(X)$, one has $A(i+1)\otimes_{\mathscr{O}}C\simeq 0$ iff $C_{q}^{\wedge}\simeq 0$ for all $q\in X_{i+1}$, which is equivalent to the condition $\Supp(C)\subseteq X\backslash X_{i+1}$ due to the faithfully-flatness of $\mathscr{O}_{q}\to\mathscr{O}^{\wedge}_{q}$ under Noetherian assumption \cite[tag 00MC]{stacks}. As $\Supp(C)$ does not contain $X_{n}\cup\cdots\cup X_{i+2}$ already, the condition is equivalent to $C\in\Perf_{\leq i}(X)$. In particular, $\Perf_{\leq i}(X)\in\Cat^{\perf}$ as a fiber of exact functors in $\Cat^{\perf}$, since fiber products of idempotent-complete $\infty$-categories over any $\infty$-category is idempotent complete \cite[Lemma 8-(ii)]{tamme}. 
\end{remarkn}

We will consider exact sequences in $\Cat^{\ex}$ of the form $\mathcal{A}\to\Perf_{\leq i}(A(T))\to\Perf_{\leq i}(A(T\sqcup\{i\}))$, where $T\subseteq[i-1]$. The following proposition treats the case of $T=\emptyset$:

\begin{proposition}\label{prop:fibern}
Let $X$ be a Noetherian scheme of finite Krull dimension $n$, and let $0\leq i\leq n$. Then, we have an exact sequence
\begin{align*}
\Perf_{\leq i -1}(X)\to \Perf_{\leq i}(X)\xrightarrow{A(i)\otimes_{\mathscr{O}}-} \Perf_{\leq i}(A(i)) ~~~~~\text{in $\Cat^{\ex}$. }
\end{align*}
\end{proposition}

Here, we set $\Perf_{\leq -1}(X) = 0$. Before giving a proof of Proposition \ref{prop:fibern}, let us consider the following two lemmas as preparation. Our goal is to describe a functor $\jmath\psh:\Ind\Perf_{\leq i}(A(i))\to\Ind\Perf_{\leq i}(X)$ which is right adjoint to the functor $\jmath\pb = A(i)\otimes_{\mathscr{O}}-$ in $\Prl_{\st}$ through Lemma \ref{lem:rightadj}. We will use Lemma \ref{lem:rightadjprep} for the proof of Lemma \ref{lem:rightadj} as well as the proof of Proposition \ref{prop:fibern} below. 

\begin{lemma}\label{lem:rightadjprep}
Let $X$ be a Noetherian scheme of finite Krull dimension $n$, and let $0\leq i\leq n$. Fix $C\in\Perf_{\leq i}(X)$. For each $q\in X$, let $h_{q}:\Spec\mathscr{O}_{q}\to X$ be the canonical flat morphism, and let $C_{q} = h_{q}\pb C\in\Perf(\mathscr{O}_{q})$. Then the followings hold:\\
(1) $A(i)\otimes_{\mathscr{O}}C$ viewed as an object of $\Mod(\mathscr{O}_{X})$ is equivalent to $\prod_{q\in S}(h_{q})\psh C_{q}$ for some finite subset $S\subseteq X_{i}$ of dimension-$i$ points. \\
(2) $A(i)\otimes_{\mathscr{O}}C$ viewed as an object of $\Mod(\mathscr{O}_{X})$ is in $\Ind\Perf_{\leq i}(X)$.
\begin{proof}
Let $S\subseteq X_{i}$ be the set of dimension-$i$ points with $C_{q}\not\simeq 0$, which is finite due to our assumption, and consider the $\mathscr{O}_{X}$-module object $\prod_{q\in S}(h_{q})\psh C_{q}$. Note that by Noetherian assumption on $X$, the canonical map $h_{q}:\Spec\mathscr{O}_{q}\to X$ is quasicompact quasiseparated, and hence $\prod_{q\in S}(h_{q})\psh C_{q}$ is in $\QCoh(X)$ \cite[tag 08D5]{stacks}. \\
\indent We first consider the affine case $X = \Spec R$ for (1). As before let $S$ be the finite set consisting of points $\mathfrak{q}\in X_{i}$ satisfying $C^{\wedge}_{\mathfrak{q}}\not\simeq 0$ (equivalently $C_{\mathfrak{q}}\not\simeq 0$ due to fully faithfulness of $R_{\mathfrak{q}}\to R^{\wedge}_{\mathfrak{q}}$). Then, the perfect complex $A(i)\otimes_{R}C = \prod_{\mathfrak{q}\in X_{i}}R^{\wedge}_{\mathfrak{q}}\otimes_{R}C$ is equivalent to $\prod_{\mathfrak{q}\in S}C^{\wedge}_{\mathfrak{q}}$ over $A(i)$ via projection $A(i)\to\prod_{\mathfrak{q}\in S}R^{\wedge}_{\mathfrak{q}}$ (as both give equivalent data in $\prod_{\mathfrak{q}\in X_{i}}\Perf(R^{\wedge}_{\mathfrak{q}})$). Now, note that each $C_{\mathfrak{q}} = R_{\mathfrak{q}}\otimes_{R}C\in\Perf(R_{\mathfrak{q}})$ satisfies $(C_{\mathfrak{q}})_{\mathfrak{p}}\simeq C_{\mathfrak{p}}\simeq 0$ for all $\mathfrak{p}\in\Spec R_{\mathfrak{q}}\backslash \{\mathfrak{q}\}$, due to assumption $C\in\Perf_{\leq i}(\Spec R)$. Hence, $C_{\mathfrak{q}}$ is canonically a perfect complex over $R^{\wedge}_{\mathfrak{q}}$, with $C_{\mathfrak{q}}\simeq C^{\wedge}_{\mathfrak{q}}$.\\
\indent Now suppose $X$ is a Noetherian scheme, possibly non-affine. For each affine open $\Spec R\subseteq X$, the restriction of $\prod_{q\in S}(h_{q})\psh C_{q}$ on $\Spec R$ viewed as an $R$-module is $\prod_{q\in S\cap\Spec R}C_{q}$, and we know that there is an equivalence $(A(i)\otimes_{\mathscr{O}}C)|_{\Spec R}\simeq \prod_{q\in S\cap\Spec R}C_{q}$ of $A(i)|_{\Spec R}$-modules obtained from the canonical projection $A(i)|_{\Spec R}\to\prod_{q\in S\cap\Spec R}\mathscr{O}^{\wedge}_{q}$ and $C^{\wedge}_{q}\simeq C_{q}$ over $\mathscr{O}_{q}$, compatible with restrictions. (Here, we are implicitly using the fact that $\Perf(A(i)|_{\Spec R})\simeq\Perf(A(i)(\Spec R))$.) Thus $\prod_{q\in S}(h_{q})\psh C_{q}$ admits an $A(i)$-module structure and is equivalent to $A(i)\otimes_{\mathscr{O}}C$ over $A(i)$ via canonical projections, hence in particular equivalent over $\mathscr{O}_{X}$ when viewed as objects in $\Mod(\mathscr{O}_{X})$.\\
\indent Finally, (2) follows from (1). Note that $(h_{q})\psh C_{q}\simeq ((h_{q})\psh\mathscr{O}_{\Spec\mathscr{O}_{q}})\otimes_{\mathscr{O}}C$ by derived projection formula \cite[tag 0B54]{stacks}, and $(h_{q})\psh\mathscr{O}_{\Spec\mathscr{O}_{q}}\in\QCoh(X)$ due to Noetherian assumption. Thus $(h_{q})\psh C_{q}$ is equivalent to a filtered colimit of the form $\colim_{k}E_{k}\otimes_{\mathscr{O}}C$, where each $E_{k}$ is in $\Perf(X)$. Thus each of $E_{k}\otimes_{\mathscr{O}}C$ is in $\Perf_{\leq i}(X)$, and we know $(h_{q})\psh C_{q}\in\Ind\Perf_{\leq i}(X)$. Since $A(i)\otimes_{\mathscr{O}}C$ is a finite product of such objects by (1), it is also in $\Ind\Perf_{\leq i}(X)$. 
\end{proof}
\end{lemma}

\begin{lemma} \label{lem:rightadj}
Let $X$ be a Noetherian scheme of dimension $n$, and fix any $0\leq i\leq n$. Consider the essentially surjective functor $\jmath\pb = A(i)\otimes_{\mathscr{O}}-:\Perf_{\leq i}(X)\to\Perf_{\leq i}(A(i))$ which induces a compact functor $A(i)\otimes_{\mathscr{O}}-:\Ind\Perf_{\leq i}(X)\to\Ind\Perf_{\leq i}(A(i))$ in $\Prl_{\st}$ still denoted by $\jmath\pb$. Also, consider the restriction of scalars functor $\rho:\Mod(A(i))\to\Mod(\mathscr{O}_{X})$ induced by $\mathscr{O}_{X}\to A(i)$. \\
\indent Then, the restriction $\Ind\Perf_{\leq i}(A(i))\to\Mod(\mathscr{O}_{X})$ of the functor $\rho$ to $\Ind\Perf_{\leq i}(A(i))$ factors through $\Ind\Perf_{\leq i}(X)$, and the resulting functor $\rho':\Ind\Perf_{\leq i}(A(i))\to\Ind\Perf_{\leq i}(X)$ is a right adjoint of $\jmath\pb$. 
\begin{proof}
Since restriction of scalars functor $\rho$ commutes with filtered colimits, it sufficies to check that for each $C\in\Perf_{\leq i}(X)$ the object $A(i)\otimes_{\mathscr{O}}C$ of $\Ind\Perf_{\leq i}(A(i))$, now viewed as an object of $\Mod(\mathscr{O}_{X})$ via $\rho$, in fact sits inside a stable subcategory $\Ind\Perf_{\leq i}(X)$. This follows from Lemma \ref{lem:rightadjprep} (2). From the already-existing adjunction $A(i)\otimes_{\mathscr{O}}-\dashv\rho:\Mod(A(i))\to\Mod(\mathscr{O}_{X})$, one knows $\rho':\Ind\Perf_{\leq i}(A(i))\to\Ind\Perf_{\leq i}(X)$ is right adjoint to $\jmath\pb$. 
\end{proof}
\end{lemma}

\begin{proof}[Proof of Proposition \ref{prop:fibern}]
We check that Ind-completion of the sequence of $\Cat^{\ex}$ in question is a split exact sequence of $\Prl_{\st}$ by following the criterion provided by Proposition \ref{prop:splitexact}. Let $\jmath\pb = A(i)\otimes_{\mathscr{O}}-:\Ind\Perf_{\leq i}(X)\to\Ind\Perf_{\leq i}(A(i))$ be an Ind-completion of the functor $A(i)\otimes_{\mathscr{O}}-:\Perf_{\leq i}(X)\to\Perf_{\leq i}(A(i))$, and let $\jmath\psh:\Ind\Perf_{\leq i}(A(i))\to\Ind\Perf_{\leq i}(X)$ be its right adjoint. As $\jmath\pb$ is compact, $\jmath\psh$ commutes with filtered colimits. We verify that $\jmath\psh$ is fully faithful, i.e., the counit map $\jmath\pb\jmath\psh(\jmath\pb C)\to \jmath\pb C$ is an equivalence for all $C\in\Perf_{\leq i}(X)$. By our description of $\jmath\psh$ as $\rho'$ in Lemma \ref{lem:rightadj}, this means we have to check the canonical map $A(i)\otimes_{\mathscr{O}}A(i)\otimes_{\mathscr{O}}C\to A(i)\otimes_{\mathscr{O}}C$ is an equivalence for $C\in\Perf_{\leq i}(X)$. As the statement is Zariski-local on $X$, we may assume that $X = \Spec R$ is affine. By Lemma \ref{lem:rightadjprep} (1), we know $A(i)\otimes_{R}C\simeq\prod_{\mathfrak{q}\in S}C^{\wedge}_{\mathfrak{q}}\simeq\prod_{\mathfrak{q}\in S}C_{\mathfrak{q}}$ as $R$-modules. We have a canonical equivalence $A(i)\otimes_{R}\prod_{\mathfrak{q}\in S}C_{\mathfrak{q}}\simeq A(i)\otimes_{R}C$ obtained as a base change of the equivalence $\prod_{\mathfrak{q}\in S}R_{\mathfrak{q}}\otimes_{R}\prod_{\mathfrak{q}\in S}C_{\mathfrak{q}}\simeq \prod_{\mathfrak{q}\in S}R_{\mathfrak{q}}\otimes_{R}C$. In fact, as products are taken over finite sets, it suffices to check that we have canonical equivalences $R_{\mathfrak{p}}\otimes_{R}\prod_{\mathfrak{q}\in S}C_{\mathfrak{q}}\simeq R_{\mathfrak{p}}\otimes_{R}C$ for each $\mathfrak{p}\in S$. As the involved base changes are flat, we can assume $C\simeq M[0]$ for some discrete finitely generated $R$-module $M$. As $M_{\mathfrak{q}}$ is supported on $\{\mathfrak{q}\}$, each $x\in M_{\mathfrak{q}}$ admits an $r>0$ with $(\mathfrak{q}R_{\mathfrak{q}})^{r}\cdot x=0$. If $\mathfrak{q}\neq\mathfrak{p}\in S$, then one can find $f\in (R\backslash\mathfrak{p})\cap\mathfrak{q}$, and $x = f^{r}x/f^{r}=0\in R_{\mathfrak{p}}\otimes_{R}M_{\mathfrak{q}}$. Thus $R_{\mathfrak{p}}\otimes_{R}M_{\mathfrak{q}}\simeq 0$ for $\mathfrak{p}\neq\mathfrak{q}$, and the claim follows.\\
\indent As the composition of the sequence is zero, it remains to compute the fiber of $\jmath\pb$. To show the fiber is equivalent to $\Ind\Perf_{\leq i-1}(X)$, we use the description of $\h\Perf_{\leq i}(X)/\h\Perf_{\leq i-1}(X)$ given as a consequence of \cite[3.24]{balmer} (see also (7) in the proof of \cite[Theorem 2]{balmer}), that the canonical triangulated functor $\h\Perf_{\leq i}(X)/\h\Perf_{\leq i-1}(X)\to\oplus_{q\in X_{i}}\h\Perf_{\{q\}}(\mathscr{O}_{q})$ exhibits the target as an idempotent completion of triangulated categories. Since the functor is precisely the image of $\Perf_{\leq i}(X)/\Perf_{\leq i-1}(X)\to\oplus_{q\in X_{i}}\Perf_{\{q\}}(\mathscr{O}_{q})$ in $\Cat^{\ex}$ by taking homotopy categories, \cite[5.15]{bgt} implies we have an exact sequence $\Ind\Perf_{\leq i-1}(X)\to\Ind\Perf_{\leq i}(X)\to\Ind(\oplus_{q\in X_{i}}\Perf_{\{q\}}(\mathscr{O}_{q}))$ of $\Prl_{\st}$. In particular, note that for each $C\in\Perf_{\leq i}(X)$, the unit map of the adjunction $C\to\oplus_{q\in X_{i}}(h_{q})\psh C_{q}\simeq \prod_{q\in S}(h_{q})\psh C_{q}$ for the second left adjoint functor agrees with the unit map $C\to\jmath\psh\jmath\pb C$ by Lemma \ref{lem:rightadjprep}. Hence, by description of the unit map in Example \ref{ex:splitexact} associated with the functor $\Ind(\Perf_{\leq i}(X))\to\Ind(\Perf_{\leq i}(X)/\Perf_{\leq i-1}(X))\simeq \Ind(\oplus_{q\in X_{i}}\Perf_{\{q\}}(\mathscr{O}_{q}))$, we know $\jmath\psh\jmath\pb C\simeq \oplus_{q\in X_{i}}(h_{q})\psh C_{q}\simeq \colim_{F\in\Perf_{\leq i-1}(X)_{/C}}\cof(F\to C)$ in $\Ind\Perf_{\leq i}(X)$. Now, consider  the fiber sequence $\colim_{F\in \Perf_{\leq i-1}(X)_{/C}}F\to C\to \jmath\psh\jmath\pb C$ obtained from taking a filtered colimt of the fiber sequences $F\to C\to\cof(F\to C)$ indexed by the filtered $\infty$-category $\Perf_{\leq i-1}(X)_{/C}$. From this, we know that the right adjoint $\imath\ush:\Ind\Perf_{\leq i}(X)\to\fib(\jmath\pb)$ of the inclusion maps compact objects $\Perf_{\leq i}(X)$ to $\Ind\Perf_{\leq i-1}(X)$. By Remark \ref{rmk:splitexact} (2), we know $\fib(\jmath\pb)\simeq\Ind\Perf_{\leq i-1}(X)$. 
\end{proof}

\begin{example}
For $i=n=\dim X$, one in particular has the exact sequence 
\begin{align*}
\Perf_{\leq n-1}(X)\to \Perf(X)\to \Perf(A(n))
\end{align*} 
in $\Cat^{\perf}$ by Proposition \ref{prop:fibern}. Note that fully faithfulness of $\jmath\psh$ in the proof of Proposition \ref{prop:fibern} for this case can also be explained through the second formula of Lemma \ref{lem:A(0)} below. 
\end{example}

We note the following lemma, which is useful in the case of $i=n$ and motivates our approach to the problem:

\begin{lemma}\label{lem:A(0)}
Let $X$ be a Noetherian scheme of finite Krull dimension $n$, and let $0\leq i_{0}<\cdots<i_{r}<n$. Then, the following canonical maps of sheaves of rings 
\begin{align*}
A(n)\otimes_{\mathscr{O}}A(i_{0},...,i_{r})\to A(n,i_{0},...,i_{r})~~~~\text{and}~~~~A(n)\otimes_{\mathscr{O}}A(n)\to A(n)
\end{align*}
are isomorphisms. 
\begin{proof}
As the statement is Zariski-local on $X$, we can assume $X = \Spec R$ is affine. Note that the set of generic points $X^{0}$ of $X$ is finite, and in particular the set of dimension $n$-points $X_{n}\subseteq X^{0}$ is finite. By the characterizing properties of sheaves of adeles, we compute 
\begin{align*}
A(n,i_{0},...,i_{r}) &\cong\prod_{\eta\in X_{n}} A((i_{0},...,i_{r},n)_{\eta},(h_{\eta})\psh h_{\eta}\pb\mathscr{O})\cong \prod_{\eta\in X_{n}}\colim_{\eta\in D(f)}A(i_{0},...,i_{r},\widetilde{R_{f}})\\
& \cong\prod_{\eta\in X_{n}}\colim_{\eta\in D(f)}\colim(A(i_{0},...,i_{r},\mathscr{O}_{\Spec R})\xrightarrow{f\cdot}A(i_{0},...,i_{r},\mathscr{O}_{\Spec R})\xrightarrow{f\cdot}\cdots) \\
& \cong \prod_{\eta\in X_{n}}\colim_{\eta\in D(f)}A(i_{0},...,i_{r},\mathscr{O})_{f}\cong A(n)\otimes_{\mathscr{O}}A(i_{0},...,i_{r}).
\end{align*}
For the second map, it suffices to check that $R_{\mathfrak{p}}\otimes_{R}R_{\mathfrak{q}}\cong 0$ for minimal prime ideals $\mathfrak{p}\neq \mathfrak{q}$. By assumption $\mathfrak{q}R_{\mathfrak{q}}$ is the unique prime ideal of $R_{\mathfrak{q}}$, hence is the nilradical of $R_{\mathfrak{q}}$. We can take $f\in (R\backslash\mathfrak{p})\cap\mathfrak{q}$, and $f^{r}=0$ in $\mathfrak{q}R_{\mathfrak{q}}$ for some $r>0$. Hence $1 = f^{r}/f^{r}=0$ in the localization $R_{\mathfrak{p}}\otimes_{R}R_{\mathfrak{q}}$ of $R_{\mathfrak{q}}$, and we have $R_{\mathfrak{p}}\otimes_{R}R_{\mathfrak{q}}\cong0$.
\end{proof}
\end{lemma}

The following proposition describes remaining exact sequences of the form $\mathcal{A}\to\Perf_{\leq i}(A(T))\to\Perf_{\leq i}(A(T\sqcup\{i\}))$ for $T\not=\emptyset$ (see Definition \ref{def:perf} for notations):

\begin{proposition}\label{prop:fiber2n}
Let $X$ be a Noetherian scheme of finite Krull dimension $n$, and let $0\leq i_{0}<\cdots<i_{r}<i\leq n$. We have an exact sequence
\begin{align*}
\Perf_{\leq i-1}(A(i_{0},...,i_{r}))\to \Perf_{\leq i}(A(i_{0},...,i_{r}))\xrightarrow{A(i,i_{0},...,i_{r})\underset{A(i_{0},...,i_{r})}{\otimes}-} \Perf_{\leq i}(A(i, i_{0},...,i_{r})) ~~~~~\text{in $\Cat^{\ex}$. }
\end{align*}
\begin{proof}
For convenience, let us denote $\underline{j}:=(i_{0},...,i_{r})$. Let $\jmath\pb = A(i,\underline{j})\otimes_{A(\underline{j})}-:\Ind\Perf_{\leq i}(A(\underline{j}))\to\Ind\Perf_{\leq i}(A(i,\underline{j}))$ be an Ind-completion of the functor $A(i,\underline{j})\otimes_{A(\underline{j})}:\Perf_{\leq i}(A(\underline{j}))\to\Perf_{\leq i}(A(i,\underline{j}))$. It is a restriction of the functor $A(i,\underline{j})\otimes_{A(\underline{j})}-:\Mod(A(\underline{j}))\to\Mod(A(i,\underline{j}))$, which we still denote by $\jmath\pb$. We would like to check an Ind-completion of the given sequence in $\Cat^{\ex}$ is a split-exact sequence of $\Prl_{\st}$ by applying Proposition \ref{prop:splitexact}. First, consider the following decomposition property:

\begin{lemma}\label{lem:keydecomposition}
The canonical map $A(\underline{j})\otimes_{\mathscr{O}}A(i)\to A(i,\underline{j})$ of $A(\underline{j})$-algebras induces an equivalence\\ $A(\underline{j})\otimes_{\mathscr{O}}A(i)\otimes_{\mathscr{O}}C\simeq A(i,\underline{j})\otimes_{\mathscr{O}}C$ in $\Mod(A(\underline{j}))$ for all $C\in\Perf_{\leq i}(X)$. 
\end{lemma}

\noindent The proof will be given below. The right adjoint $\jmath\psh:\Mod(A(i,\underline{j}))\to\Mod(A(\underline{j}))$ of the functor $\jmath\pb$ is given as the restriction of scalars functor induced by $A(\underline{j})\to A(i,\underline{j})$, and hence commutes with filtered colimits. By restriction to $\Ind\Perf_{\leq i}(A(i,\underline{j}))$ it induces $\jmath\psh:\Ind\Perf_{\leq i}(A(i,\underline{j}))\to\Ind\Perf_{\leq i}(A(\underline{j}))$, since for $C\in\Perf_{\leq i}(X)$, one has $\jmath\psh(A(i,\underline{j})\otimes_{\mathscr{O}}C)\simeq A(\underline{j})\otimes_{\mathscr{O}}\left(A(i)\otimes_{\mathscr{O}}C\right)$ by Lemma \ref{lem:keydecomposition}, with $A(i)\otimes_{\mathscr{O}}C\in\Ind\Perf_{\leq i}(X)$ by Lemma \ref{lem:rightadj}. Thus, it is still a right adjoint of $\jmath\pb = A(i,\underline{j})\otimes_{A(\underline{j})}-:\Ind\Perf_{\leq i}(A(\underline{j}))\to\Ind\Perf_{\leq i}(A(i,\underline{j}))$. \\
\indent Using this description of the right adjoint $\jmath\psh:\Ind\Perf_{\leq i}(A(i,\underline{j}))\to\Ind\Perf_{\leq i}(A(\underline{j}))$, we check that this functor $\jmath\psh$ is fully faithful, i.e., the counit map for the associated adjunction is an equivalence. It suffices to verify the canonical equivalence $A(i,\underline{j})\otimes_{A(\underline{j})}A(i,\underline{j})\otimes_{\mathscr{O}}C\simeq A(i,\underline{j})\otimes_{\mathscr{O}}C$ in $\Mod(A(i,\underline{j}))$ for $C\in\Perf_{\leq i}(X)$. From Proposition \ref{prop:fibern}, we know $A(i)\otimes_{\mathscr{O}}A(i)\otimes_{\mathscr{O}}C\overset{\sim}{\to} A(i)\otimes_{\mathscr{O}}C$. Base change to $A(i,\underline{j})$ over $A(i)$ gives the canonical equivalence $A(i,\underline{j})\otimes_{\mathscr{O}}A(i)\otimes_{\mathscr{O}}C\overset{\sim}{\to} A(i,\underline{j})\otimes_{\mathscr{O}}C$. As the source is equivalent to $A(i,\underline{j})\otimes_{A(\underline{j})}A(\underline{j})\otimes_{\mathscr{O}}A(i)\otimes_{\mathscr{O}}C$, again Lemma \ref{lem:keydecomposition} gives a desired equivalence. Before computing the fiber of $\jmath\pb$, let us give a proof of the Lemma:

\begin{proof}[Proof of lemma \ref{lem:keydecomposition}]
For $i = n$, we have an isomorphism $A(\underline{j})\otimes_{\mathscr{O}}A(n)\simeq A(n,\underline{j})$ by Lemma \ref{lem:A(0)}, hence by tensoring with $C$ the result follows. Now we give a proof which works for the general case. As the statement is Zariski-local on $X$, we can further assume $X = \Spec R$ is affine. We have to check $A(\underline{j})\otimes_{R}A(i)\otimes_{R}C\simeq A(i,\underline{j})\otimes_{A(i)}A(i)\otimes_{R}C$. By Lemma \ref{lem:rightadjprep} (1), $A(i)\otimes_{R}C\simeq \prod_{\mathfrak{q}\in S}C^{\wedge}_{\mathfrak{q}}$ for some finite set $S\subseteq X_{i}$, and each $C^{\wedge}_{\mathfrak{q}}$ is in $\Perf_{\{\mathfrak{q}\}}(R^{\wedge}_{\mathfrak{q}})$. In particular, $C_{\mathfrak{q}}$ is canonically a perfect module over $R^{\wedge}_{\mathfrak{q}}$, and $C_{\mathfrak{q}}\simeq C^{\wedge}_{\mathfrak{q}}$. Thus, we have to prove $A(\underline{j})\otimes_{R}\prod_{\mathfrak{q}\in S}C_{\mathfrak{q}}\simeq A(i,\underline{j})\otimes_{A(i)}\prod_{\mathfrak{q}}C_{\mathfrak{q}}$, and since the product is over a finite set, we are reduced to proving that $A(\underline{j})\otimes_{R}C_{\mathfrak{q}}\simeq A(i,\underline{j})\otimes_{A(i)}C_{\mathfrak{q}}$, i.e.,
\begin{align*}
\left(A(i)\otimes_{R}R_{\mathfrak{q}}\right)\otimes_{R_{\mathfrak{q}}}C_{\mathfrak{q}}\simeq \left(A(i,\underline{j})\otimes_{A(i)}R^{\wedge}_{\mathfrak{q}}\right)\otimes_{R^{\wedge}_{\mathfrak{q}}}C_{\mathfrak{q}}~~~~~~\text{over $A(\underline{j})$,}
\end{align*}
for $\mathfrak{q}\in X_{i}$ and $C_{\mathfrak{q}}\simeq C^{\wedge}_{\mathfrak{q}}\in\Perf_{\{\mathfrak{q}\}}(R^{\wedge}_{\mathfrak{q}})$. As $A(i,\underline{j})\cong\prod_{q'\in X_{i}}\lim_{s}A_{s\overline{q'}}(i,\underline{j})$ by construction \cite[p. 65]{intro}, the perfect module $A(i,\underline{j})\otimes_{A(i)}C_{\mathfrak{q}}\in\Perf(A(i,\underline{j}))$ is equivalent to $\lim_{s}A_{s\overline{\mathfrak{q}}}(i,\underline{j})\otimes_{R^{\wedge}_{\mathfrak{q}}}C_{\mathfrak{q}}$, and we have to prove that 
\begin{align}
\left(A(\underline{j})\otimes_{R}R_{\mathfrak{q}}\right)\otimes_{R_{\mathfrak{q}}}C_{\mathfrak{q}}\simeq (\lim_{s}A_{s\overline{\mathfrak{q}}}(i,\underline{j}))\otimes_{R_{\mathfrak{q}}^{\wedge}}C_{\mathfrak{q}}~~~~~~\text{over $A(\underline{j})$.} \label{formula:local}
\end{align}
\begin{lemma}
$(A(\underline{j})\otimes_{R}R_{\mathfrak{q}})^{\wedge}_{\mathfrak{q}}\isomto\lim_{s}A_{s\overline{\mathfrak{q}}}(i,\underline{j})$, where the completion is taken at the ideal $\mathfrak{q}(A(\underline{j})\otimes_{R}R_{\mathfrak{q}})$. 
\begin{proof}
For each $s$, one has $(A(\underline{j})\otimes_{R}R_{\mathfrak{q}})/\mathfrak{q}^{s}(A(\underline{j})\otimes_{R}R_{\mathfrak{q}})\cong (A(\underline{j})/\mathfrak{q}^{s}A(\underline{j}))\otimes_{R}R_{\mathfrak{q}}\cong (A(\underline{j})\otimes_{R}R/\mathfrak{q}^{s})\otimes_{R}R_{\mathfrak{q}}$. Now, observe that by viewing $R/\mathfrak{q}^{s}$ as a coherent $R$-module, one has $A(\underline{j})\otimes_{R}R/\mathfrak{q}^{s}\cong A(\underline{j},R/\mathfrak{q}^{s}) = A(\underline{j},\imath\psh\mathscr{O}_{s\overline{\mathfrak{q}}})\cong A_{s\overline{\mathfrak{q}}}(\underline{j})$. For the last isomorphism, note that $\underline{j} = \left(\underline{j}\cap S^{\text{red}}_{r-1}(V(\mathfrak{q}^{s}))\right)\coprod \{(p_{0},...,p_{r})\in \underline{j}~|~p_{r}\in\Spec R\backslash V(\mathfrak{q}^{s})\}$, so $A(\underline{j},\imath\psh\mathscr{O}_{s\overline{\mathfrak{q}}})\cong A\left(\underline{j}\cap S^{\text{red}}_{r-1}(V(\mathfrak{q}^{s})),\imath\psh\mathscr{O}_{s\overline{\mathfrak{q}}}\right)\times A\left(\{(p_{0},...,p_{r})\in \underline{j}~|~p_{r}\in\Spec R\backslash V(\mathfrak{q}^{s})\},\imath\psh\mathscr{O}_{s\overline{\mathfrak{q}}}\right)\cong A_{\Spec R/\mathfrak{q}^{s}}(\underline{j})\times 0$, using \cite[Proposition 2.1.5]{huber}. Hence, we can continue the chain of canonical isomorphisms as 
\begin{align*}
(A(\underline{j})\otimes_{R}R_{\mathfrak{q}})/\mathfrak{q}^{s}(A(\underline{j})\otimes_{R}R_{\mathfrak{q}}) & \cong A_{s\overline{\mathfrak{q}}}(\underline{j})\otimes_{R/\mathfrak{q}^{s}}R/\mathfrak{q}^{s}\otimes_{R}R_{\mathfrak{q}}\\
&\cong  A_{s\overline{\mathfrak{q}}}(\underline{j})\otimes_{R/\mathfrak{q}^{s}}R_{\mathfrak{q}}/\mathfrak{q}^{s}R_{\mathfrak{q}} \cong A_{s\overline{\mathfrak{q}}}(\underline{j})\otimes_{R/\mathfrak{q}^{s}}\Frac(R/\mathfrak{q}^{s})\cong A_{s\overline{\mathfrak{q}}}(i,\underline{j}).
\end{align*}  
These isomorphisms (for each $s$) are compatible with each other, and hence induce an isomorphism between limits. 
\end{proof}
\end{lemma}

Hence combined with the lemma below (applied to $A = A(\underline{j})\otimes_{R}R_{\mathfrak{q}}$), we have the canonical equivalence (\ref{formula:local}), finishing the proof. Note that $A(\underline{j})$ is flat over $R$ \cite[Lemma 1.10]{adelic}.

\begin{lemma}
Let $R = (R,\mathfrak{m},\kappa)$ be a Noetherian local ring, $R^{\wedge}$ be its completion at $\mathfrak{m}$, and $A$ be a flat $R$-algebra. Also, let $A^{\wedge} = \lim_{s}A/\mathfrak{m}^{s}A$, which is canonically an algebra over $A$ and $R^{\wedge}$. Then the canonical map of exact functors $A\otimes_{R}(-)\to A^{\wedge}\otimes_{R^{\wedge}}(-)(\simeq A^{\wedge}\otimes_{R}-)$ from $\Perf_{\{\mathfrak{m}\}}(R)\simeq\Perf_{\{\mathfrak{m}\}}(R^{\wedge})$ to $\Mod(A)$ is an equivalence. 
\begin{proof}
By induction on the number of nonzero homotopy modules, it suffices to check the equivalence $A\otimes_{R}M\simeq A^{\wedge}\otimes_{R}M$ for discrete finitely generated $R$-modules $M$ supported on the point $\{\mathfrak{m}\}$. More precisely, for each $C\in\Perf_{\{\mathfrak{m}\}}(R)$ one can apply exact functors on a truncation fiber sequence of the form $\pi_{k}(C)\to C\to\tau_{<k}C$, with $\tau_{<k}C$ having strictly less number of nonvanishing homotopy modules. By assumption on $M$, there is an $r>0$ with $\mathfrak{m}^{r}M\simeq0$. Thus, by applying exact functors on fiber sequences $\mathfrak{m}^{i}M\to M\to M/\mathfrak{m}^{i}M$ ($0\leq i\leq r$), one knows it suffices to verify the equivalence for $R$-modules $\mathfrak{m}^{i-1}M/\mathfrak{m}^{i}M$, or more generally for finite $R/\mathfrak{m}$-modules viewed as $R$-modules. Hence, it suffices to verify the equivalence $A\otimes_{R}\kappa\simeq A^{\wedge}\otimes_{R}\kappa$. Note that by \cite[tag 0AGW]{stacks} or \cite[Theorem 0.1]{yek}, $A^{\wedge}$ is still flat over $R$, and the involved base changes are underived. Thus, both sides are canonically equivalent to $A/\mathfrak{m}A$. 
\end{proof}
\end{lemma}
This finishes the proof of Lemma \ref{lem:keydecomposition}. 
\end{proof}  
It remains to describe the fiber of $\jmath\pb$. Let $A(\underline{j})\otimes_{\mathscr{O}}C\in\Perf_{\leq i}(A(\underline{j}))$, where $C\in\Perf_{\leq i}(X)$. By Lemma \ref{lem:keydecomposition}, its unit map $A(\underline{j})\otimes_{\mathscr{O}}C\to\jmath\psh\jmath\pb(A(\underline{j})\otimes_{\mathscr{O}}C)\simeq A(i,\underline{j})\otimes_{\mathscr{O}}C$ is equivalent to the image of the unit map $C\to A(i)\otimes_{\mathscr{O}}C$ of Proposition \ref{prop:fibern} by $A(\underline{j})\otimes_{\mathscr{O}}-$. Again by Proposition \ref{prop:fibern}, Remark \ref{rmk:splitexact} (1), and exactness of $A(\underline{j})\otimes_{\mathscr{O}}-$, we have a fiber sequence $A(\underline{j})\otimes_{\mathscr{O}}F\to A(\underline{j})\otimes_{\mathscr{O}}C\to \jmath\psh\jmath\pb(A(\underline{j})\otimes_{\mathscr{O}}C)$ in $\Ind\Perf_{\leq i}(A(\underline{j}))$, where $F\in\Ind\Perf_{\leq i-1}(X)$. Thus, a right adjoint $\imath\ush:\Ind\Perf_{\leq i}(A(\underline{j}))\to\fib(\jmath\pb)$ maps $\Perf_{\leq i}(A(\underline{j}))$ to $\Ind\Perf_{\leq i-1}(A(\underline{j}))$, and by Remark \ref{rmk:splitexact} (2), we know $\fib(\jmath\pb)\simeq \Ind\Perf_{\leq i-1}(A(\underline{j}))$.
\end{proof}
\end{proposition}

\begin{example}
For $i = n = \dim X$ and for $0\leq i_{1}\leq\cdots\leq i_{r}<n$, we in particular have an exact sequence
\begin{align*}
\Perf_{\leq n-1}(A(i_{0},...,i_{r}))\to \Perf_{\leq n}(A(i_{0},...,i_{r}))\xrightarrow{A(n,i_{0},...,i_{r})\underset{A(i_{0},...,i_{r})}{\otimes}-} \Perf_{\leq n}(A(n, i_{0},...,i_{r})) ~~~~~\text{in $\Cat^{\ex}$. }
\end{align*}
Note that fully faithfulness of a right adjoint of (an Ind-completion of) the functor $A(n,i_{0},...,i_{r})\underset{A(i_{0},...,i_{r})}{\otimes}-$ can be explained by the second isomorphism in Lemma \ref{lem:A(0)}. By Remark \ref{rmk:idem}, applying any localizing invariant $E:\Cat^{\ex}\to\mathcal{T}$ to above exact sequence yields the fiber sequence 
\begin{align*}
E(\Perf_{\leq n-1}(A(i_{0},...,i_{r})))\to E(A(i_{0},...,i_{r}))\to E(A(n,i_{0},...,i_{r}))
\end{align*} 
in a stable $\infty$-category $\mathcal{T}$.
\end{example}

\begin{remarkn}
In fact, we can verify that the fiber of the functor $\jmath\pb = A(i,\underline{j})\otimes_{A(\underline{j})}-$ in Proposition \ref{prop:fiber2n} on compact objects $\Perf_{\leq i}(A(\underline{j}))$ is $\Perf_{\leq i-1}(A(\underline{j}))$ via direct computation. We have a canonical morphism (i.e., a square) from $\Perf_{\leq i}(X)\to\Perf_{\leq i}(A(i))$ of Proposition \ref{prop:fibern} to $\Perf_{\leq i}(A(\underline{j}))\to \Perf_{\leq i}(A(i, \underline{j}))$, whose component functors are essentially surjective. Thus, it sufficies to prove the following:
\begin{proposition}\label{prop:kerneldirectcomp}
Let $X$ be a Noetherian scheme of finite Krull dimension $n$, and take $i$ and $\underline{j}$ as in Proposition \ref{prop:fiber2n}. For $C\in\Perf_{\leq i}(X)$, the vanishing $A(i,\underline{j})\otimes_{\mathscr{O}}C\simeq 0$ implies $A(i)\otimes_{\mathscr{O}}C\simeq 0$. 
\end{proposition}

\begin{lemma}\label{lem:rationalization2}
Proposition \ref{prop:kerneldirectcomp} holds for $i = n$. In other words, $A(n,\underline{j})\otimes_{\mathscr{O}}C\simeq 0$ implies $A(n)\otimes_{\mathscr{O}}C\simeq 0$ for $C\in\Perf(X)$. 
\begin{proof}
The question is Zariski-local on $X$, so we can assume $X = \Spec R$. We proceed by induction on $\dim X$. The case of $\dim X = 0$ is tautological, as the only possible choice of the sheaf is $A(0)$. Suppose $\dim X>0$, and let $\underline{j} = (i_{0},...,i_{r})$. By Lemma \ref{lem:A(0)}, the assumption equivalently says $A(n)\otimes_{R}\left(A(\underline{j})\otimes_{R}C\right)\simeq 0$, so $A(\underline{j})\otimes_{R}C$ vanishes at each points of $X_{n}$. Fix any $\eta\in X_{n}$. We can take an affine open subset where $A(\underline{j})\otimes_{R}C$ vanishes, since it is concentrated in finitely many degrees and $A(\underline{j})$ is flat over $R$. Thus we can assume our $C\in\Perf(R)$ satisfies $A(\underline{j})\otimes_{R}C\simeq 0$ (where $i_{r}<n$). By base change to $A(n-1,\underline{j})$ over $A(\underline{j})$ (if $i_{r}<n-1$), we can further assume $i_{r}=n-1$. We are reduced to checking that this condition, together with the induction hypothesis, imply $C_{\eta}\simeq 0$. \\
\indent From $A(\underline{j})\simeq \prod_{\mathfrak{q}\in(\Spec R)_{i_{r}}}\lim_{s}A_{s\overline{\mathfrak{q}}}(\underline{j})$ (e.g., \cite[p. 65]{intro}), we know $0\simeq \lim_{s}A_{s\overline{\mathfrak{q}}}(\underline{j})\otimes_{R}C$ for all $\mathfrak{q}\in (\Spec R)_{i_{r}}$. In particular, $0\simeq A_{\overline{\mathfrak{q}}}(\underline{j})\otimes_{R}C\simeq A_{R/\mathfrak{q}}(\underline{j})\otimes_{R/\mathfrak{q}}\left(R/\mathfrak{q}\otimes_{R}C\right)$ holds. By the induction hypothesis applied to $\Spec R/\mathfrak{q}$, we know $R/\mathfrak{q}\otimes_{R}C$ satisfies $0\simeq \Frac(R/\mathfrak{q})\otimes_{R/\mathfrak{q}}\left(R/\mathfrak{q}\otimes_{R}C\right)\simeq \kappa(R^{\wedge}_{\mathfrak{q}})\otimes_{R^{\wedge}_{\mathfrak{q}}}R^{\wedge}_{\mathfrak{q}}\otimes_{R}C\simeq \kappa(R^{\wedge}_{\mathfrak{q}})\otimes_{R^{\wedge}_{\mathfrak{q}}}C^{\wedge}_{\mathfrak{q}}$. (Here, $\kappa(R^{\wedge}_{\mathfrak{q}})$ stands for the residue field of $R^{\wedge}_{\mathfrak{q}}$.) By the derived Nakayama lemma (Remark \ref{rmk:variousperfect}), we have $C^{\wedge}_{\mathfrak{q}}\simeq 0\in\Perf(R^{\wedge}_{\mathfrak{q}})$ for all $\mathfrak{q}\in(\Spec R)_{i_{r}}$. By the Noetherian hypothesis $R_{\mathfrak{q}}\to R^{\wedge}_{\mathfrak{q}}$ is faithfully flat \cite[tag 00MC]{stacks}, and we in particular know ($C_{\mathfrak{q}}\simeq 0$, and) $\Supp(C)\subseteq \Spec R\backslash (\Spec R)_{i_{r}}$. Thus, the closed subset $\Supp(C)\subseteq \Spec R$ ($C$ is perfect) should not contain codimension $1$ points of $\overline{\eta}$, and we know $\eta\notin\Supp(C)$, i.e., $C_{\eta}\simeq 0$. 
\end{proof}
\end{lemma}

\begin{proof}[Proof of Proposition \ref{prop:kerneldirectcomp}]
As the statement is Zariski-local on $X$, we can assume $X = \Spec R$. From the assumption $0\simeq A(i,\underline{j})\otimes_{A(i)}\left(A(i)\otimes_{R}C\right)$, we have $0\simeq \lim_{s}A_{s\overline{\mathfrak{q}}}(i,\underline{j})\otimes_{R^{\wedge}_{\mathfrak{q}}}C^{\wedge}_{\mathfrak{q}}\simeq \lim_{s}A_{s\overline{\mathfrak{q}}}(i,\underline{j})\otimes_{R}C$ for all $\mathfrak{q}\in X_{i}$. In particular, $A_{R/\mathfrak{q}}(i,\underline{j})\otimes_{R}C\simeq 0$, or equivalently $A_{R/\mathfrak{q}}(i,\underline{j})\otimes_{R/\mathfrak{q}}\left(R/\mathfrak{q}\otimes_{R}C\right)\simeq 0$ for all $\mathfrak{q}\in(\Spec R)_{i}$. By Lemma \ref{lem:rationalization2} applied to $\Spec R/\mathfrak{q}$, we have $0\simeq A_{R/\mathfrak{q}}(i)\otimes_{R/\mathfrak{q}}\left(R/\mathfrak{q}\otimes_{R}C\right)\simeq \Frac(R/\mathfrak{q})\otimes_{R}C\simeq R^{\wedge}_{\mathfrak{q}}/\mathfrak{q}R^{\wedge}_{\mathfrak{q}}\otimes_{R^{\wedge}_{\mathfrak{q}}}R^{\wedge}_{\mathfrak{q}}\otimes_{R}C\simeq \kappa(R^{\wedge}_{\mathfrak{q}})\otimes_{R^{\wedge}_{\mathfrak{q}}}C^{\wedge}_{\mathfrak{q}}$. By the derived Nakayama lemma (Remark \ref{rmk:variousperfect}), we know $C^{\wedge}_{q}\simeq 0$ for all $q\in X_{i}$, i.e., $A(i)\otimes_{\mathscr{O}}C\simeq 0$. 
\end{proof}

\end{remarkn}

By combining Proposition \ref{prop:fibern} and Proposition \ref{prop:fiber2n}, we have the following:

\begin{proposition}\label{prop:exactseqadele}
Let $X$ be a Noetherian scheme of finite Krull dimension $n$, and let $0\leq i\leq n$. Then for each $T\in\mathcal{P}([i-1])$, we have an exact sequence 
\begin{align*}
\Perf_{\leq i-1}(A(T))\to \Perf_{\leq i}(A(T))\xrightarrow{A(T\sqcup\{i\})\underset{A(T)}{\otimes}-} \Perf_{\leq i}(A(T\sqcup\{i\})) ~~~~~\text{in $\Cat^{\ex}$. }
\end{align*}
\begin{proof}
The case of $T=\emptyset$ follows from Proposition \ref{prop:fibern}, and the remaining case of $T\not=\emptyset$ follows from Proposition \ref{prop:fiber2n}. 
\end{proof}
\end{proposition}

\subsection{Adelic descent for localizing invariants}
\label{subsec:adelicdescent}

Let $X$ be a Noetherian scheme of finite Krull dimension $n$. Recall that we have semi-cosimplicial and cubical diagrams $A^{\sbullet}_{\text{red}}(X):=A^{\sbullet}_{\text{red}}(X,\mathscr{O}_{X}):(\Delta_{s})_{\leq n}\to\CAlg(\mathscr{O}_{X})^{\heartsuit}$ and $A(-):\mathcal{P}([n])\to\CAlg(\mathscr{O}_{X})^{\heartsuit}$ of $\mathscr{O}_{X}$-algebras (Remarks \ref{rmk:functoriality} and \ref{rmk:cubicaladeles}). By setting $A_{\text{red}}^{-\infty}(X) := \mathscr{O}_{X}$, we can view $A_{\text{red}}^{\sbullet}(X)$ as an augmented semi-cosimplicial diagram, and after composition with $\Perf(-)$ we have an augmented semi-cosimplicial diagram $\Perf(A^{\sbullet}_{\text{red}}(X))$ in $\Cat^{\perf}$. Likewise, we have an $n$-cubical diagram $\Perf(A(-))$ in $\Cat^{\perf}$. For each $0\leq i\leq n$, we also consider the $n$-cube $\Perf_{\leq i}(A(-))$ in $\Cat^{\ex}$ induced as a subfunctor of $\Perf(A(-))$. Note that $\Perf_{\leq i}(A(-))$ can be regarded as an $i$-cubical diagram after restriction to $\mathcal{P}([i])$, since $\Perf_{\leq i}(X)\to\Perf(A(i_{0},...,i_{r}))$ factors through $\Perf(A(i_{r}))$ and hence $\Perf_{\leq i}(A(i_{0},...,i_{r}))\simeq 0$ for $i_{r}>i$. By further composing these diagrams with a localizing invariant $E:\Cat^{\ex}\to\mathcal{T}$, we obtain (augmented) semi-cosimplicial and cubical diagrams $E(A^{\sbullet}_{\text{red}}(X))$, $E(A(-))$, and $E(\Perf_{\leq i}(A(-)))$ in a stable $\infty$-category $\mathcal{T}$. 

\begin{theorem}\label{thm:adelicdescent}
Let $X$ be a Noetherian scheme of finite Krull dimension $n$, and let $E:\Cat^{\ex}\to\mathcal{T}$ be a localizing invariant valued in a stable $\infty$-category $\mathcal{T}$.\\
(1) For each $0\leq i\leq n$, the $n$-cubical diagram $E(\Perf_{\leq i}(A(-))):\N\mathcal{P}([n])\to\mathcal{T}$ is a limit diagram. In particular the $n$-cubical diagram $E(A(-)):\N\mathcal{P}([n])\to\mathcal{T}$ is a limit diagram, and we have an equivalence $E(X)\simeq\lim_{0\leq i_{0}<\cdots<i_{r}\leq n}E(A(i_{0},...,i_{r}))$ in $\mathcal{T}$. \\
(2) The (truncated) augmented semi-cosimplicial diagram $E(A^{\sbullet}_{\text{red}}(X)):\N((\Delta_{s})_{+})_{\leq n}\to\mathcal{T}$ is a limit diagram, and we have an equivalence $E(X)\simeq \lim_{[r]\in(\Delta_{s})_{\leq n}}E(A^{r}_{\text{red}}(X))$ in $\mathcal{T}$. 
\begin{proof}
We prove (1) through induction on $i$. By Proposition \ref{prop:fibern} the underlying functor of the $0$-cubical diagram $\Perf_{\leq 0}(X)\to\Perf_{\leq 0}(A(0))$ is an equivalence, and we in particular have $i=0$ case by applying $E$. Suppose $0<i\leq n$, and consider the $n$-cubical diagrams $\Perf_{\leq i}(A(-))$ and $E(\Perf_{\leq i}(A(-)))$. In order to check $E(\Perf_{\leq i}(A(-)))$ is a limit diagram, it suffices to check the $i$-cubical diagram obtained by a restriction to $\mathcal{P}([i])$ is a limit diagram, as images of the other vertices are zero. Now, consider the decomposition $\mathcal{P}([i]) = \mathcal{P}([i-1])\coprod\left(\mathcal{P}([i-1])\sqcup\{i\}\right)$ and view the $i$-cube $\Perf_{\leq i}(A(-))|_{\mathcal{P}([i])}$ as a morphism $\Perf_{\leq i}(A(-))|_{\mathcal{P}([i-1])}\to\Perf_{\leq i}(A(i,-))|_{\mathcal{P}([i-1])}$ of $(i-1)$-cubical diagrams, and similarly for $E(\Perf_{\leq i}(A(-)))$. By applying $E$ to the exact sequences of Proposition \ref{prop:exactseqadele}, we know $\fib\left(E(\Perf_{\leq i}(A(-)))|_{\mathcal{P}([i-1])}\to E(\Perf_{\leq i}(A(i,-)))|_{\mathcal{P}([i-1])}\right)\simeq E(\Perf_{\leq i-1}(A(-)))|_{\mathcal{P}([i-1])}$. By induction hypothesis this $(i-1)$-cubical diagram is a limit diagram, and hence by Proposition \ref{prop:fiberlimcrit} we know the original $i$-cubical diagram $E(\Perf_{\leq i}(A(-)))|_{\mathcal{P}([i])}$ is a limit diagram, i.e, $E(\Perf_{\leq i}(A(-)))$ is a limit diagram. This establishes (1), and in particular for $i=n$ we have $E(X)\simeq \lim_{T\in\mathcal{P}([n])\backslash\emptyset}E(A(T))$ by Remark \ref{rmk:idem}. By Corollary \ref{cor:fiberlim} and Remark \ref{rmk:cubicaladeles}, we know $E(A^{\sbullet}_{\text{red}}(X))$ is also a limit diagram, and have $E(X)\simeq\lim_{T\in\mathcal{P}([n])\backslash\emptyset}E(A(T))\simeq \lim_{[r]\in(\Delta_{s})_{\leq n}}E(A^{r}_{\text{red}}(X))$ by (1). 
\end{proof}
\end{theorem}

\textsc{Department of Mathematics, University of Toronto, Toronto, ON, Canada}\\
\indent \textit{E-mail address}: \texttt{khsato@math.utoronto.ca}

\end{document}